\documentclass{amsart}
\usepackage{amsmath, amsthm, amssymb}

\usepackage[a4paper,width=150mm,top=25mm,bottom=25mm,centering]{geometry}

\usepackage{hyperref} 
\usepackage{enumitem}
\usepackage{graphicx} 
\usepackage{xcolor}
\usepackage{comment}

\usepackage{tikz}
\usepackage{pgf}
\usepackage{tikzscale}
\usepackage{tikz-cd}
\usetikzlibrary{decorations.pathreplacing,decorations.markings}
\usetikzlibrary{arrows, knots, external, positioning}

\usepackage{url}
\usepackage{caption} 

\newtheorem{theorem}{Theorem}[section]
\newtheorem{conjecture}[theorem]{Conjecture}
\newtheorem{lemma}[theorem]{Lemma}
\newtheorem{proposition}[theorem]{Proposition}
\newtheorem{corollary}[theorem]{Corollary}
\newtheorem{question}[theorem]{Question}

\theoremstyle{definition}
\newtheorem{definition}[theorem]{Definition}
\newtheorem{remark}[theorem]{Remark}

\newtheorem{notation}[theorem]{Notation}

\renewcommand{\Bbb}{\mathbb}

\newcommand{\LK}{\mathcal{L}_{\textup{Dehn}}}
\newcommand{\LS}{\mathcal{L}_{\pi_1}}
\newcommand{\LB}{\mathcal{L}_{b_1}}

\begin{document}

\title[On lens spaces bounding smooth 4-manifolds with $b_2=1$]{On lens spaces bounding smooth 4-manifolds with $\boldsymbol{b_2=1}$}

\author{Woohyeok Jo}
\address{Department of Mathematical Sciences, Seoul National University, Seoul 08826,
 Republic of Korea}
\email{koko1681@snu.ac.kr}

\author{Jongil Park}
\address{Department of Mathematical Sciences and Research Institute of Mathematics, Seoul National University, Seoul 08826, Republic of Korea}
\email{jipark@snu.ac.kr}

\author{Kyungbae Park}
\address{Department of Mathematics, Kangwon National University, Gangwon 24341,
 Republic of Korea}
\email{kyungbaepark@kangwon.ac.kr}

\thanks{}


\begin{abstract}
    We study which lens spaces can bound smooth 4-manifolds with second Betti number one under various topological conditions. Specifically, we show that there are infinite families of lens spaces that bound compact, simply-connected, smooth 4-manifolds with second Betti number one, yet cannot bound a 4-manifold consisting of a single 0-handle and 2-handle. Additionally, we establish the existence of infinite families of lens spaces that bound compact, smooth 4-manifolds with first Betti number zero and second Betti number one, but cannot bound simply-connected 4-manifolds with second Betti number one. The construction of such 4-manifolds with lens space boundaries is motivated by the study of rational homology projective planes with cyclic quotient singularities.
\end{abstract}

\maketitle

\section{Introduction}
In this paper, all manifolds are assumed to be compact, connected, and oriented unless stated otherwise. For relatively prime integers $p>q>0$, the lens space $L(p,q)$ is the $3$-manifold obtained from $S^3$ by $-p/q$-surgery along the unknot. An interesting topic in low-dimensional topology is determining which lens spaces can bound a smooth 4-manifold with specific topological properties.

For $4$-manifolds with minimal Betti numbers, a remarkable result by Lisca \cite{Lisca-2007} provides a complete classification of lens spaces that can bound a rational homology 4-ball (a smooth 4-manifold whose homology groups with rational coefficients are isomorphic to those of the 4-ball, or equivalently, with vanishing first and second Betti numbers, i.e., $b_1=b_2=0$). Interestingly, Lisca's classification results in the confirmation of the \emph{slice-ribbon conjecture}, a well-known open problem in knot theory, for 2-bridge knots. 

As a next step, one may consider lens spaces that bound a smooth 4-manifold with $b_1=0$ and $b_2=1$. (Note that the first and third Betti numbers, $b_1$ and $b_3$, of a smooth 4-manifold with boundary, can always be reduced to zero without changing the boundary. See \cite[Corollary 5.3.5]{Gompf-Stipsicz-1999}. Also note that, for a 4-manifold whose boundary is a lens space, $b_1=0$ implies $b_3=0$.) The simplest such 4-manifolds, in terms of handle decomposition, are those consisting of a single $0$- and $2$-handle. By reversing the orientation if necessary, we may restrict our attention to $4$-manifolds with positive definite intersection forms, i.e., $b_2=b_2^+=1$. This leads naturally to the following question, which has already been completely resolved by Greene.

\begin{question}\label{question:lens_space_realization} Which lens spaces can bound a compact, oriented, smooth $4$-manifold $W$ with $b_2(W)=b_2^+(W)=1$ which is built from a single $0$- and $2$-handle? 
\end{question}

Question \ref{question:lens_space_realization} is equivalent to the \textit{lens space realization problem} \cite{Greene-2013}, which asks which lens spaces can arise from a positive integer surgery along a knot in $S^3$. By definition, a \textit{lens space knot} is a knot in $S^3$ that admits such a surgery. Certain families of lens space knots were listed by Berge \cite{Berge-2018}, and the \textit{Berge conjecture} posits that these constitute the complete list of lens space knots. Greene proved that if a lens space is obtained by a positive integer surgery along a knot in $S^3$, then it can also be obtained from a positive integer surgery along one of Berge's knots. This result resolves the lens space realization problem and confirms `half' of the Berge conjecture; see Section \ref{subsec:lens_space_realization} for further details.

The following two generalizations of Question \ref{question:lens_space_realization} are discussed in \cite[Section 1.6]{Greene-2013}.

\begin{question}\label{question:filling} Which lens spaces can bound a compact, oriented, smooth $4$-manifold $W$ with $b_2(W)=b_2^+(W)=1$ and $\pi_1(W)=1$?  
\end{question}

\begin{question}\label{question:filling_2} Which lens spaces can bound a compact, oriented, smooth $4$-manifold $W$ with $b_2(W)=b_2^+(W)=1$ and $b_1(W)=0$?  
\end{question}
We remark that if lens spaces are replaced by connected sums of lens spaces in Question \ref{question:filling}, the question becomes closely related to the Montgomery-Yang problem (see \cite{Kollar-2008, Jo-Park-Park-2025-01}). We also note that the answer to Question \ref{question:filling} is known in the topological category. For relatively prime integers $p>q>0$, the lens space $L(p,q)$ bounds a simply-connected, \textit{topological} $4$-manifold with $b_2=b_2^+=1$ if and only if $-q$ is a quadratic residue modulo $p$ \cite{Boyer-1986}. However, there are many lens spaces $L(p,q)$ that cannot bound a simply-connected, smooth $4$-manifold with $b_2=b_2^+=1$, even if $-q$ is a quadratic residue modulo $p$ \cite{Fukuhara-1984, Saeki-1989}.

For convenience, let us define collections of lens spaces as follows. Let $\LK$ denote the collection of all lens spaces that bound a smooth $4$-manifold with $b_2=b_2^+=1$ built from a single $0$- and $2$-handle (Question \ref{question:lens_space_realization}). Next, let $\LS$ denote the collection of all lens spaces that bound a smooth $4$-manifold with $b_2=b_2^+=1$ and $\pi_1=1$ (Question \ref{question:filling}). Finally, let $\LB$ denote the collection of all lens spaces that bound a smooth $4$-manifold with $b_2=b_2^+=1$ and $b_1=0$ (Question \ref{question:filling_2}). Observe that we have the inclusions $\LK \subset \LS \subset \LB$, and the collection $\LK$ is completely classified by Greene.

As noted in \cite[Section 1.6]{Greene-2013}, the lens space $L(17,2)$ serves as an example of a lens space in $\LS \setminus \LK$, demonstrated as follows: Tange established the existence of a positive definite $2$-handle cobordism from the Brieskorn homology sphere $\Sigma(2,5,7)$ to $L(17,2)$ \cite[Section 1.3]{Tange-2018}. It is well known that $\Sigma(2,5,7)$ bounds a contractible 4-manifold \cite{Akbulut-Kirby-1979, Casson-Harer-1981}. By attaching this contractible 4-manifold to the 2-handle cobordism along $\Sigma(2,5,7)$, he obtained a simply-connected 4-manifold with boundary $L(17,2)$. However, the lens space $L(17,2)$ does not belong to $\LK$ by Greene's classification. Additional candidates expected to lie in $\LS \setminus \LK$ are provided in \cite[Proposition 1.15]{Tange-2018} through similar reasoning. In \cite{Ballinger-2022}, Ballinger also presented an infinite family of lens spaces in $\LS$ and asserted that this family is not contained in $\LK$. His approach involved finding embeddings of linear plumbed $4$-manifolds of length $n-1$ in $\#n\mathbb{CP}^2$; the complement of such an embedding is a smooth $4$-manifold with $b_2=b_2^+=1$, with a lens space boundary, and is simply-connected under certain conditions. We will explore these families further in Section \ref{appendix}.

The goal of this article is to present infinitely many examples of lens spaces in $\LS\setminus \LK$ and $\LB\setminus \LS$, respectively. To construct 4-manifolds with $b_1=0$ and $b_2=b_2^+=1$ bounded by a lens space, we employ a different method from those described above. In fact, as we discuss in Section \ref{appendix}, some of the expected families from Tange cannot be obtained using our approach. Our construction is inspired by the study of rational homology projective planes (normal projective complex surfaces whose Betti numbers match those of the complex projective plane $\mathbb{CP}^2$) with quotient singularities in algebraic surface theory. In the works of the authors \cite{Jo-Park-Park-2025-01, Jo-Park-Park-2024-02}, we utilized the study of topological or smooth 4-manifolds to address classification problems of rational homology projective planes. In this paper, however, we use rational homology projective planes to resolve questions in low-dimensional topology.

Specifically, to show that a lens space $L(p,q)$ belongs to $\LB$ or $\LS$, we explicitly construct a rational homology projective plane with a unique cyclic singularity of type $(p,p-q)$ (see Section \ref{subsec:qhcp2} for details). This singularity admits a neighborhood that is homeomorphic to the cone on the lens space $L(p,p-q)$, and the complement of the cone neighborhood is a smooth 4-manifold with $b_1=0$, $b_2=b_2^+=1$, and boundary $L(p,q)$. The first homology group and the fundamental group of the resulting $4$-manifold can be computed through a relatively simple calculation. 

In the course of classifying rational homology projective planes in \cite{Jo-Park-Park-2024-02}, we discovered numerous examples that yield interesting fillings of lens spaces. By extending these examples to infinite families and taking into account various conditions related to $\LK$ and $\LS$, we obtained the lens spaces that appear in our main theorems below. Our first result is to present the following two-parameter infinite family of lens spaces contained in $\LS\setminus \LK$. To show that these lens spaces are obstructed from being in $\LK$, we use Greene's argument on lattice embeddings and the concept of \textit{changemaker vectors} (Definition \ref{def:changemaker}).

\begin{theorem}\label{thm:filling2} For nonnegative integers $m,n\geq 0$, let $p_{m,n}>q_{m,n}>0$ be relatively prime integers determined by the continued fraction
\[
\frac{p_{m,n}}{p_{m,n}-q_{m,n}}=\left[2,2,2,n+2,m,\left[2\right]^m\right].
\]
Then, for each $m\geq 8$ and $n\geq 7$ with $m$ odd, the lens space $L(p_{m,n},q_{m,n})$ bounds a compact, oriented, smooth $4$-manifold with $\pi_1=1$ and $b_2=b_2^+=1$, but does not bound such a $4$-manifold built from a single $0$- and $2$-handle.
\end{theorem}

Using a similar construction, we also obtained a two-parameter infinite family of lens spaces contained in $\LB\setminus\LS$. For these lens spaces, we obstruct their inclusion in $\LS$ through a detailed lattice embedding argument based on Donaldson's diagonalization theorem (Corollary \ref{cor:orthogonal_complement}).

\begin{theorem}\label{thm:filling1} For nonnegative integers $m,n\geq 0$, let $p_{m,n}>q_{m,n}>0$ be relatively prime integers determined by the continued fraction
\[
    \frac{p_{m,n}}{p_{m,n}-q_{m,n}}=\left[m+2, \left[2\right]^{m+2}, n+2, \left[2\right]^{m+2}, m+2\right].
\]
Then, for each $m\geq 6$ and $n\geq 0$ with $n\neq 2,4,6$, the lens space $L(p_{m,n},q_{m,n})$ bounds a compact, oriented, smooth $4$-manifold with $b_1=0$ and $b_2=b_2^+=1$, but does not bound such a $4$-manifold with $H_1=0$ (and hence also does not bound one with $\pi_1=1$). 
\end{theorem}

\begin{remark}
    In \cite[Section 1.6]{Greene-2013}, Greene proposed that $L(10,1)$ is a lens space in $\LB\setminus \LS$; however, the argument provided does not establish this. To support the inclusion of $L(10,1)$ in $\LB$, he considered the union of a positive definite single 2-handle cobordism from $L(10,1)$ to the Brieskorn sphere $\Sigma(2,3,7)$ and a rational homology 4-ball bounded by $\Sigma(2,3,7)$. However, the boundary of this union is $-L(10,1)=L(10,9)$, rather than $L(10,1)$.

    Nonetheless, it is known that there exists a rational homology projective plane with a unique singularity of type $A_9$ (i.e., a cyclic singularity of type $(10,9)$) \cite{Hwang-Keum-Ohashi-2015}, which confirms that $L(10,1)$ is indeed in $\LB$.

    To show that $L(10,1)\notin \LS$, Greene applied a result of Kervaire and Milnor \cite{KM-1961}. Alternatively, one can employ Heegaard Floer $d$-invariants to establish that $L(10,1)\notin \LS$; if such a simply connected 4-manifold exists, then it must be a spin smooth 4-manifold with $b_2=b_2^+=1$ and boundary $L(10,1)$, contradicting that the $d$-invariants for the two spin structures of $L(10,1)$ are $1/4$ and $-9/4$; see \cite[Section 2.3.3]{Jo-Park-Park-2025-01}. 
\end{remark}

\begin{remark} There are many lens spaces that are not contained in $\LB$. For example, for $n>1$, the lens space $L(n,1)$ is contained in $\LB$ only if $n$ can be expressed as a sum of two squares, as shown through an application of Donaldson's diagonalization theorem (Theorem \ref{thm:Donaldson}). It is well known that a positive integer $n>1$ can be written as a sum of two squares if and only if, in its prime factorization, no prime $p$ such that $p\equiv 3 \mod{4}$ appears with odd multiplicity.

In general, determining the minimal second Betti number of definite fillings for a given lens space is a challenging problem (see \cite{AMP-2025}).
\end{remark}

One may also consider the collection $\mathcal{L}_{H_1}$ of lens spaces that bound a smooth $4$-manifold with $b_2=b_2^+=1$ and $H_1=0$. (Here $H_1$ denotes the first homology group with integer coefficients.) It is clear that $\LS \subset \mathcal{L}_{H_1}\subset \LB$. Theorem \ref{thm:filling1} shows that the inclusion $\mathcal{L}_{H_1}\subset \LB$ is proper. However, determining whether the inclusion $\LS\subset\mathcal{L}_{H_1}$ is proper is expected to be quite daunting, due to the lack of known obstructions for satisfying the $\pi_1=1$ condition when the lens space bounds a smooth 4-manifold with $H_1=0$.

\begin{question}
    Is there a lens space that bounds a compact, oriented, smooth $4$-manifold $W$ with $b_2(W)=b_2^+(W)=1$ and $H_1(W;\mathbb{Z})=0$ but not such a $4$-manifold with $\pi_1=1$?
\end{question}

\subsection*{Acknowledgments} The authors thank all members of the 4-manifold topology group at Seoul National University (SNU) for their invaluable comments throughout this work. We express our gratitude to Joshua Greene for posing the problem that motivated this work and bringing Ballinger's result to our attention. We would also like to thank the referee for their time and valuable comments which have contributed to the improvement of this manuscript.  
Jongil Park was supported by the National Research Foundation of Korea (NRF) grant funded by the Korean government 
(No.2020R1A5A1016126, RS-2023-00301976 and RS-2024-00392067). He is also affiliated with the Research Institute of Mathematics at SNU. Kyungbae Park was supported by NRF grant funded by the Korean government (RS-2025-24523511).

\section{Preliminaries}\label{sec:preliminaries}
In this section, we briefly review some background knowledge on the main obstruction and construction relevant to this article. For relatively prime integers $p>q>0$, the lens space $L(p,q)$ is the oriented 3-manifold obtained by $-p/q$-surgery along the unknot in $S^3$.  Expand $p/q$ into its uniquely determined Hirzebruch-Jung continued fraction as follows:
\[ 
\frac{p}{q}=[a_1,\dots,a_\ell]:=a_1-\frac{1}{a_2-\displaystyle\frac{1}{\cdots-\displaystyle\frac{1}{a_\ell}}} \ ~~(a_i\geq 2).
\]
It is well-known that $L(p,q)$ is the boundary of the negative definite plumbed 4-manifold $X(p,q)$ constructed from the linear graph in Figure \ref{fig:plumbing_graph_of_X(p,q)} (see \cite[Exercise 5.3.9(b)]{Gompf-Stipsicz-1999}). 
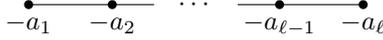
\begin{figure}[!th]
\centering
\begin{tikzpicture}[scale=1.1]
\draw (-2,0) node[circle, fill, inner sep=1.2pt, black]{};
\draw (-1,0) node[circle, fill, inner sep=1.2pt, black]{};
\draw (1,0) node[circle, fill, inner sep=1.2pt, black]{};
\draw (2,0) node[circle, fill, inner sep=1.2pt, black]{};

\draw (-2,0) node[below]{$-a_1$};
\draw (-1,0) node[below]{$-a_2$};
\draw (1,0) node[below]{$-a_{\ell-1}$};
\draw (2,0) node[below]{$-a_\ell$};

\draw (0,0) node{$\cdots$};

\draw (-2,0)--(-1,0) (-1,0)--(-0.5,0) (0.5,0)--(1,0)  (1,0)--(2,0) ;
\end{tikzpicture}
\caption{The plumbing graph of $X(p,q)$.}
\label{fig:plumbing_graph_of_X(p,q)}
\end{figure}

Since $L(p,p-q)$ is homeomorphic to the orientation reversal $-L(p,q)$ of $L(p,q)$, we deduce that $-L(p,q)$ is the boundary of the negative definite 4-manifold $X(p,p-q)$.

The following lemma is well-known, see \cite[Exercise 5.3.13(f),(g)]{Gompf-Stipsicz-1999} for example.

\begin{lemma}\label{lem:linking_form} Let $Y$ be a rational homology $3$-sphere, and $W$ a compact, oriented, topological $4$-manifold with $H_1(W;\Bbb Z)=0$ and $\partial W=Y$. If $A$ is any matrix for the intersection form of $W$, then $H_1(Y;\Bbb Z)$ is isomorphic to the cokernel of $A\colon\Bbb Z^{b_2(W)}\to \Bbb Z^{b_2(W)}$.
\end{lemma}

In particular, if $Y$ is a lens space, then we have the following corollary.

\begin{corollary}\label{cor:linking_form} Suppose that the lens space $L(p,q)$ bounds a compact, oriented, topological $4$-manifold $W$ with $H_1(W;\mathbb{Z})=0$ and $b_2(W)=b_2^+(W)=1$. Then the intersection form of $W$ is represented by the $1\times 1$ matrix $(p)$.     
\end{corollary}
\begin{proof} As $b_2(W)=b_2^+(W)=1$, the intersection form of $W$ is represented by $(n)$ for a uniquely determined positive integer $n$. By Lemma \ref{lem:linking_form}, $H_1(\partial W;\mathbb{Z})=H_1(L(p,q);\mathbb{Z})$ is isomorphic to $\mathbb{Z}_n$, implying that $p=n$. \end{proof}

\subsection{Donaldson's Diagonalization Theorem and Lattice Embedding}
For a compact, oriented $4$-manifold $X$ and its intersection form
\[Q_X\colon H_2(X;\Bbb Z)/\textrm{Tor}\times H_2(X;\Bbb Z)/\textrm{Tor}\to \Bbb Z,\] 
we simply denote the intersection lattice $(H_2(X;\Bbb Z)/\textrm{Tor}, Q_X)$ by $Q_X$. For a positive integer $n$, let $\{e_1,\dots,e_n\}$ be the standard basis for $\Bbb Z^n$. We denote by $-\Bbb Z^n$ the standard negative definite lattice $(\Bbb Z^n, \langle\cdot,\cdot\rangle)$ given by $\langle e_i,e_j\rangle = -\delta_{i,j}$, with $\delta_{i,j}$ being the Kronecker delta. 

Donaldson's diagonalization theorem states that there is a significant constraint on the intersection forms of closed, oriented, smooth, definite $4$-manifolds. 

\begin{theorem}[Donaldson's Diagonalization Theorem, {\cite{Donaldson-1983,Donaldson-1987}}]\label{thm:Donaldson} If the intersection form $Q_X$ of a closed, oriented, smooth $4$-manifold $X$ is negative definite, then $Q_X$ is isomorphic to $-\Bbb Z^n$, where $n=b_2(X)=b_2^-(X)$.     
\end{theorem}
Recall the following observation, which can be obtained from algebraic topology.

\begin{proposition}[{\cite[Lemma 2.4]{AMP-2025}}]\label{prop:orthogonal_complement} Let $Y$ be an oriented $3$-manifold with $H^1(Y;\mathbb{Z})=0$ which is the boundary of compact, oriented $4$-manifolds $X_1$ and $X_2$ with $H_1(X_1;\Bbb Z)=0$. If $X$ is a closed, oriented $4$-manifold obtained by $X:=X_1\cup_Y (-X_2)$, then the inclusions $X_1,-X_2\hookrightarrow X$ induce an embedding of lattices \[
  \iota\colon Q_{X_1}\oplus (-Q_{X_2}) \to Q_X
\]
such that $\iota(-Q_{X_2})$ is the orthogonal complement of $\iota(Q_{X_1})$ in $Q_X$.     
\end{proposition}

These constraints on the intersection lattices of 4-manifolds have been used to provide conditions for a smooth $4$-manifold with a specified boundary $3$-manifold. For example, the aforementioned results of Lisca \cite{Lisca-2007} and Greene \cite{Greene-2013} are based on analyzing lattice embeddings. In particular, for a smooth 4-manifold $W$ with $H_1(W;\mathbb{Z})=0$, $b_2(W)=b_2^+(W)=1$, and a lens space boundary, we have the following corollary.

\begin{corollary}\label{cor:orthogonal_complement} If $L(p,q)$ bounds a compact, oriented, smooth $4$-manifold $W$ with $H_1(W;\mathbb{Z})=0$ and $b_2(W)=b_2^+(W)=1$, then there exists an embedding $\iota\colon Q_{X(p,q)}\hookrightarrow -\mathbb{Z}^{b_2(X(p,q))+1}$ of lattices such that the generator (which is uniquely determined up to sign) of the orthogonal complement of $\iota\left(Q_{X(p,q)}\right)$ in $-\mathbb{Z}^{b_2(X(p,q))+1}$ has square $-p$.   
\end{corollary}
\begin{proof} Note that the intersection form of $W$ is represented by the matrix $(p)$ by Corollary \ref{cor:linking_form}. Now, consider the closed, oriented, smooth $4$-manifold $X:=(-W)\cup_\partial X(p,q)$, which is negative definite with $b_2(X)=b_2(X(p,q))+1$. By Theorem \ref{thm:Donaldson}, the intersection form $Q_X$ of $X$ is isomorphic to $-\mathbb{Z}^{b_2(X(p,q))+1}$. The result now follows immediately from Proposition \ref{prop:orthogonal_complement}.    
\end{proof}

\subsection{The Lens Space Realization Problem}\label{subsec:lens_space_realization}
In this subsection, we briefly recall Greene's result on the lens space realization problem \cite{Greene-2013} and provide a description of the collection $\LK$, which will be used to demonstrate that certain families of lens spaces are not contained in $\LK$.

\begin{notation}\label{notation:knot_surgery} For a knot $K\subset S^3$ and a positive integer $p\in \mathbb{Z}_{>0}$, let $S^3_p(K)$ denote the 3-manifold obtained from $S^3$ by $p$-surgery along $K$, and let $W_p(K)$ denote the 4-manifold obtained by attaching a $p$-framed $2$-handle to $D^4$ along $K\subset S^3=\partial D^4$. Note that $\partial W_p(K)=S^3_p(K)$.  
\end{notation}

A knot in $S^3$ that admits a lens space surgery is called a \textit{lens space knot}. Several families of lens space knots were discovered by Berge \cite{Berge-2018}, and the following \textit{Berge conjecture} posits that all lens space knots are, in fact, Berge's knots.

\begin{conjecture}[Berge Conjecture, {\cite[Problem 1.78]{Kirby-1995}, \cite[Conjecture 1.1]{Greene-2013}}]\label{conj:Berge} If an integer surgery along a knot $K$ in $S^3$ produces a lens space, then it must arise from Berge’s construction.     
\end{conjecture}

Given a lens space knot $K\subset S^3$ such that $p$-surgery along $K$ yields the lens space $L(p,q)$, there is a corresponding \textit{dual knot} $K'\subset S^3_p(K)=L(p,q)$, which is the core of the surgery solid torus. Reversing the surgery, a \textit{negative} integer surgery on $L(p,q)$ along $K'$ recovers $S^3$.

Berge's knots are \textit{doubly primitive} knots, which are knots that lie on a Heegaard surface of genus two for $S^3$ and represent a primitive element in the fundamental group of each handlebody. (Conversely, every doubly primitive knot in $S^3$ is a Berge knot \cite[Theorem 1.3]{Greene-2013}.) The dual to a doubly primitive knot is an example of a \textit{simple knot}, of which there is exactly one in each homology class in $L(p,q)$ \cite{Berge-2018}. Thus, the dual to a Berge knot is determined by its homology class. The so-called \textit{Berge list}, summarized in \cite[Section 6.2]{Rasmussen-2007} and \cite[Section 1.2]{Greene-2013}, describes the knots in $L(p,q)$ that are dual to Berge's knots. To describe a dual knot that yields a negative $S^3$ surgery, select a positive integer $k$ and determine the corresponding positive integer $p$ from the list. The value $k \mod p$ represents the homology class of the dual knot in $H_1(L(p,q);\mathbb{Z})\cong \mathbb{Z}_p$, where we have $q \equiv -k^2\mod p$. 

Greene has proved that lens spaces in the Berge list are precisely all lens spaces that can arise from a positive integer surgery along a knot in $S^3$. This result resolves the lens space realization problem and confirms the lens spaces that appear in the Berge conjecture.

\begin{theorem}[{\cite[Theorem 1.3]{Greene-2013}}]\label{thm:Greene} Let $K$ be a knot in $S^3$ and $p$ a positive integer. If $S^3_p(K)$ is a lens space, then there exists a Berge knot $B$ in $S^3$ such that $S^3_p(B)\cong S^3_p(K)$.    
\end{theorem}

It also follows that the collection $\LK$ consists of those lens spaces $L(p,q)$ such that $(p,q)$ (or $(p,q')$, where $qq'\equiv 1\mod p$) appears in Berge's list. An alternative and useful description of Berge’s list is provided through the following definition.

\begin{definition}[{\cite[Definition 1.5]{Greene-2013}}]\label{def:changemaker} A vector $(\sigma_0,\dots,\sigma_n)\in \mathbb{Z}^{n+1}$ with $1=\sigma_0\leq \sigma_1\leq \cdots \leq \sigma_n$ is called a \textit{changemaker} if, for any integer $k$ with $0\leq k\leq \sigma_0+\cdots+\sigma_n$, there exists a subset $A\subset \{0,\dots,n\}$ such that $\sum_{i\in A}\sigma_i=k$. Equivalently, $\sigma_i\leq \sigma_0+\cdots+\sigma_{i-1}+1$ for each $i=1,\dots,n$. 
\end{definition}

\begin{theorem}[{\cite[Theorem 1.6]{Greene-2013}, \cite[Theorem 3.3]{Greene-2015}}]\label{thm:changemaker} Suppose that $p$-surgery along a knot $K\subset S^3$ produces the lens space $L(p,q)$, where $p$ is a positive integer. Then there exists a full rank lattice embedding \[
\iota\colon Q_{X(p,q)}\oplus \left(-Q_{W_p(K)}\right) \hookrightarrow -\mathbb{Z}^{b_2(X(p,q))+1}
\]
such that the image of a generator of $H_2(-W_p(K);\mathbb{Z})\cong \mathbb{Z}$ is a changemaker with square $-p$.
\end{theorem}

By Proposition \ref{prop:orthogonal_complement}, $\iota\left(-Q_{W_p(K)}\right)$ and $\iota(Q_{X(p,q)})$ are orthogonal complements to each other in $-\mathbb{Z}^{b_2(X(p,q))+1}$. Hence, Theorem \ref{thm:changemaker} implies that the orthogonal complement of $\iota\left(Q_{X(p,q)}\right)$ in $-\mathbb{Z}^{b_2(X(p,q))+1}$ is generated by a changemaker. Furthermore, $Q_{X(p,q)}$ embeds as the orthogonal complement of a changemaker in $-\mathbb{Z}^{b_2(X(p,q))+1}$. The following theorem shows that the converse also holds, offering another characterization of Berge's list.

\begin{theorem}[{\cite[Theorem 1.7]{Greene-2013}}]\label{thm:changemaker2} At least one of the pairs $(p,q)$, $(p,q')$, where $qq'\equiv 1 \mod p$, appears on Berge's list if and only if $Q_{X(p,q)}$ embeds as the orthogonal complement of a changemaker in $-\mathbb{Z}^{b_2(X(p,q))+1}$.
\end{theorem}

When dealing with an infinite family of lens spaces, described in terms of continued fraction expansions of their parameters, it is generally more convenient to use the changemaker condition (as in Propositions \ref{prop:changemaker_application} and \ref{prop:changemaker_application_2}), rather than verifying directly whether these lens spaces appear in Berge’s list.

\subsection{Rational Homology Projective Planes with Unique Cyclic Singularity}\label{subsec:qhcp2}
A normal projective complex surface $S$ whose Betti numbers $b_i(S)$ match those of the complex projective plane $\mathbb{CP}^2$ is called a \textit{rational homology  projective plane} (or a $\mathbb{Q}$-homology $\mathbb{CP}^2$). A quotient singularity $p\in \textup{Sing}(S)$ is called \textit{cyclic} if the reduced exceptional divisor of its minimal resolution consists of rational curves, and it has the weighted dual graph as shown in Figure \ref{fig:cyclic_sing} with $b_i\geq 2$. In this case, $p$ is said to be \textit{cyclic of type $(n,a)$} where $\frac{n}{a}=[b_1,\dots,b_m]$ and $n>a>0$ are relatively prime. In fact, the germ $(S,p)$ is locally analytically isomorphic to $(\mathbb{C}^2/C_{n,a},0)$, where $C_{n,a}\subset \textup{GL}(2,\mathbb{C})$ is the cyclic subgroup of order $n$ generated by the matrix \[
\begin{bmatrix}
    e^{2\pi i/ n} & 0 \\ 0 & e^{2\pi i a/n}
\end{bmatrix}.
\]

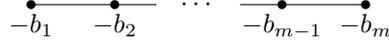
\begin{figure}[!t]
\centering
\begin{tikzpicture}[scale=1.1]
\draw (-2,0) node[circle, fill, inner sep=1.2pt, black]{};
\draw (-1,0) node[circle, fill, inner sep=1.2pt, black]{};
\draw (1,0) node[circle, fill, inner sep=1.2pt, black]{};
\draw (2,0) node[circle, fill, inner sep=1.2pt, black]{};

\draw (-2,0) node[below]{$-b_1$};
\draw (-1,0) node[below]{$-b_2$};
\draw (1,0) node[below]{$-b_{m-1}$};
\draw (2,0) node[below]{$-b_m$};

\draw (0,0) node{$\cdots$};

\draw (-2,0)--(-1,0) (-1,0)--(-0.5,0) (0.5,0)--(1,0)  (1,0)--(2,0) ;
\end{tikzpicture}
\caption{The weighted dual graph of the cyclic singularity of type $(n,a)$.}
\label{fig:cyclic_sing}
\end{figure}

Suppose that $S$ is a $\mathbb{Q}$-homology $\mathbb{CP}^2$ with a unique singularity $p\in \textup{Sing}(S)$ which is cyclic of type $(n,a)$. Then the link of $S$ at $p$ is the lens space $L(n,a)$, and $p$ has a neighborhood that is homeomorphic to the cone on $L(n,a)$. Removing this cone neighborhood from $S$ yields a smooth $4$-manifold with $b_1=0$ and $b_2=b_2^+=1$, whose boundary is the lens space $L(n,n-a)$. 

Therefore, given relatively prime integers $p>q>0$, one way to construct a smooth $4$-manifold with $b_1=0$, $b_2=b_2^+=1$ and boundary $L(p,q)$ is to build a $\mathbb{Q}$-homology $\mathbb{CP}^2$ with a unique singularity that is cyclic of type $(p,p-q)$. However, such a $\mathbb{Q}$-homology $\mathbb{CP}^2$ does not exist for arbitrary values of $p$ and $q$. A necessary condition for its existence is provided by the following proposition: 

\begin{proposition}\label{prop:D_square} Let $S$ be a $\mathbb{Q}$-homology $\mathbb{CP}^2$ with a unique singularity that is cyclic of type $(n,a)$. Then \[
D:=n\cdot \left(9-3m+\sum_{i=1}^m b_i -2 +\frac{a+a'+2}{n}\right)
\]
must be a square number (including zero), where $\displaystyle\frac{n}{a}=[b_1,\dots,b_m]$ with $b_i\geq 2$ and $a'$ is the unique inverse of $a$ modulo $n$ such that $0<a'<n$. 
\end{proposition}
\begin{proof}
    Let $p\in S$ be the unique singularity of $S$, and let $f\colon \tilde{S}\to S$ denote the minimal resolution of $S$. Then we have the following numerical equivalence between the canonical divisors: \[
        K_{\tilde{S}} \equiv f^*K_S - D_p,
    \]
    where $D_p$ is an effective $\mathbb{Q}$-divisor supported on $f^{-1}(p)$. Therefore, we have \[
        K_S^2=K_{\tilde{S}}^2-D_p^2.
    \]
    Now $D_p^2$ can be computed by applying the adjunction formula \cite[Lemma 3.6]{Hwang-Keum-2011-2}: \[
        D_p^2=2m-\sum_{i=1}^m b_i +2 -\frac{a+a'+2}{n},
    \]
    and $K_{\tilde{S}}^2$ can be computed by using Noether's formula: \[
    K_{\tilde{S}}^2 = 12-e(\tilde{S})=12-(m+3)=9-m.
    \]
    (Note that $p_g(\tilde{S})=q(\tilde{S})=0$ since quotient singularities are rational.) It follows that the number $D$ is precisely equal to $nK_S^2$. If $K_S$ is numerically trivial, then clearly $D=0$. On the other hand, if $K_S$ is not numerically trivial, then $D=nK_S^2$ must be a nonzero square number by \cite[Lemma 3.3(3)]{Hwang-Keum-2011-2}.
\end{proof}

There are various known $\mathbb{Q}$-homology $\mathbb{CP}^2$'s with unique cyclic singularities. In particular, such $\mathbb{Q}$-homology $\mathbb{CP}^2$'s whose canonical divisors are anti-ample have been classified \cite{Kojima-1999}. For cases where the canonical divisor is not anti-ample, see \cite[Theorem 1.5]{Hwang-Keum-2012} for example.

A standard method for constructing a $\mathbb{Q}$-homology $\mathbb{CP}^2$ with a unique cyclic quotient singularity, which we will use, is as follows: Start with a configuration of rational curves in the complex projective plane $\mathbb{CP}^2$ or in the Hirzebruch surface $\Sigma_n$ for some degree $n\geq 0$. Then, blow up successively at the intersection points to obtain a linear chain of rational curves, where the length and self-intersection numbers are determined by the desired singularity. The length of this linear chain must be one less than the second Betti number $b_2$ of the resulting surface. Finally, by contracting this linear chain, we obtain a rational homology projective plane with the desired cyclic singularity.

\section{Lens Spaces in $\LS\setminus\LK$}\label{sec:filling2}
In this section, we present a family of lens spaces in $\LS\setminus\LK$, i.e., those that can bound a compact, simply-connected, smooth 4-manifold with $b_2=b_2^+=1$ but do not bound such a $4$-manifold obtained from a single $0$- and $2$-handle. 

For integers $m\geq 3$ and $n\geq 1$, define \[
    p_{m,n}:=4m^2n+5m^2-4m-4\quad\text{and}\quad q_{m,n}:=m^2n+m^2-m-1.\]
We have \[
\frac{p_{m,n}}{q_{m,n}}=\left[ 5,\left[2\right]^{n-1},3,\left[2\right]^{m-3},m+2 \right],\]
where $\left[2\right]^k$ denotes a sequence of $k$ repetitions of the entry $2$, i.e., $[\underbrace{2,\dots,2}_{k}]$. 

We first show that $L(p_{m,n},q_{m,n})\in\LS$ for each $m\geq 3$ and $n\geq 1$ with $m$ odd, by constructing the desired $\mathbb{Q}$-homology $\mathbb{CP}^2$'s.

\begin{proposition}\label{prop:construction2} For each $m\geq 3$ and $n\geq 1$ with $m$ odd, the lens space $L(p_{m,n},q_{m,n})$ bounds a compact, oriented, simply-connected, smooth $4$-manifold $W_{m,n}$ with $b_2(W_{m,n})=b_2^+(W_{m,n})=1$.     
\end{proposition}

\begin{proof} We show that there exists a $\mathbb{Q}$-homology $\mathbb{CP}^2$ having a unique cyclic singularity of type $(p_{m,n},p_{m,n}-q_{m,n})$ (Section \ref{subsec:qhcp2}). Note that \[
p_{m,n}-q_{m,n}=3m^2n+4m^2-3m-3
\]
and that \[
\frac{p_{m,n}}{p_{m,n}-q_{m,n}}=\left[2,2,2,n+2,m,\left[2\right]^m\right].
\]

Consider a configuration of the union of the $(-n)$-section and two fibers in the Hirzebruch surface $\Sigma_n$ of degree $n$, as shown in Figure \ref{fig:enter-label3} (a). Blowing-up each of the two marked intersection points twice results in the configuration of rational curves depicted in Figure \ref{fig:enter-label3} (b). Next, blow up the marked intersection point $m-2$ times, followed by a blow-up at each of the two final $(-1)$-curves. This yields the configuration of rational curves shown in Figure \ref{fig:enter-label3} (c).

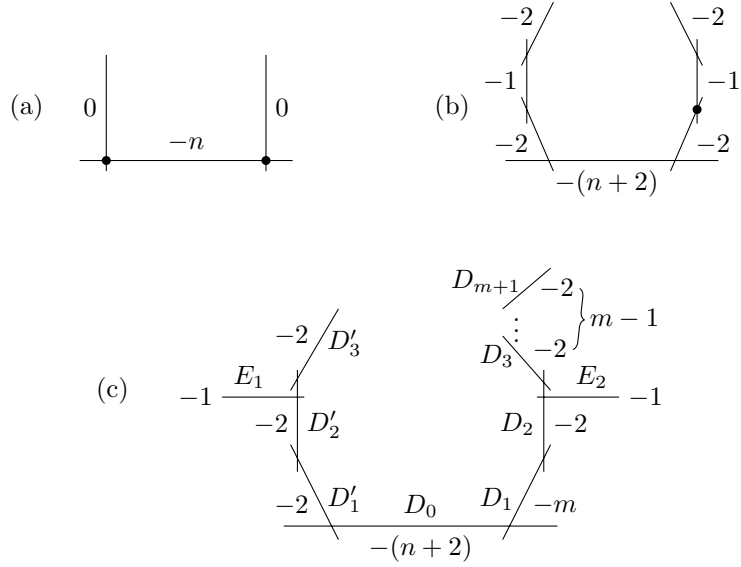
\begin{figure}[h]
\centering
\begin{tikzpicture}[scale=.7]
    \node at (-3,1) {(a)};
    \draw (-2,0)--(2,0) (1.5,-0.2)--(1.5,2)  (-1.5,-0.2)--(-1.5,2) ;
    \draw (0,0) node[above]{$-n$};
    \draw (-1.5,1) node[left]{$0$};
    \draw (1.5,1) node[right]{$0$};
    \draw (1.5,0) node[circle, fill, inner sep=1.2pt, black]{};
    \draw (-1.5,0) node[circle, fill, inner sep=1.2pt, black]{};

    \begin{scope}[shift={(2,0)}]
    \node at (3,1) {(b)};
    \draw (4,0)--(8,0) (7.1,-0.2)--(7.7,1.2) (4.3,1.2)--(4.9,-0.2) (7.7,1.8)--(7.1,3) (4.3,1.8)--(4.9,3) (4.4,0.7)--(4.4,2.3) (7.6,0.7)--(7.6,2.3);
    \draw (5.9,0) node[below]{$-(n+2)$};
    \draw (4.6,0.3) node[left]{$-2$};
    \draw (7.4,0.3) node[right]{$-2$};
    \draw (4.4,1.5) node[left]{$-1$};
    \draw (7.6,1.5) node[right]{$-1$};
    \draw (4.7,2.7) node[left]{$-2$};
    \draw (7.3,2.7) node[right]{$-2$};
    \draw (7.6,0.9666667) node[circle, fill, inner sep=1.2pt, black]{};
    \end{scope}
\end{tikzpicture}
\vspace{20pt}

\begin{tikzpicture}[scale=.9]
    \node at (-5.5,2) {(c)};
        \draw (-3,0)--(1,0) (0.2,-0.2)--(0.9,1.2) (-2.2,-0.2)--(-2.9,1.2) (-2.8,0.8)--(-2.8,2.3) (-2.9,2)--(-2.2,3.2) (0.8,0.8)--(0.8,2.3) (0.9,2)--(0.2,2.8) (0.2,3.2)--(0.9,3.8) (-2.7,1.9)--(-3.9,1.9) (0.7,1.9)--(1.9,1.9);
        \draw (0.4,3.1) node{$\vdots$};
        \draw (-2.5,0.3) node[left]{$-2$};
\draw (-1,0) node[below]{$-(n+2)$};
\draw (1.9,1.9) node[right]{$-1$};
\draw (-3.9,1.9) node[left]{$-1$};
\draw (0.5,0.3) node[right]{$-m$};
\draw (0.8,1.5) node[right]{$-2$};
\draw (-2.8,1.5) node[left]{$-2$};
\draw (-2.5,2.8) node[left]{$-2$};
\draw (0.6,3.5) node[right]{$-2$};
\draw (0.5,2.6) node[right]{$-2$};

\draw (-1,0) node[above]{$D_0$};
\draw (-2.5,0.4) node[right]{$D_1'$};
\draw (-2.8,1.5) node[right]{$D_2'$};
\draw (-3.5,1.9) node[above]{$E_1$};
\draw (-2.5,2.7) node[right]{$D_3'$};
\draw (0.5,0.4) node[left]{$D_1$};
\draw (0.8,1.5) node[left]{$D_2$};
\draw (1.5,1.9) node[above]{$E_2$};
\draw (0.6,3.6) node[left]{$D_{m+1}$};
\draw (0.5,2.5) node[left]{$D_3$};
\draw [decorate,decoration={brace,amplitude=3pt},xshift=5pt]
	(1.15,3.5) -- (1.05,2.6) node [black,midway,xshift=18pt,yshift=0pt] 
	{$m-1$};
    \end{tikzpicture}
    \caption{Configurations over $m+4$ blow-ups from Hirzebruch surface $\Sigma_n$.}
    \label{fig:enter-label3}
\end{figure}

Let $\tilde{S}_{m,n}$ denote the resulting surface, obtained by blowing up in total $m+4$ times from $\Sigma_n$. Label the rational curves as shown in Figure \ref{fig:enter-label3} (c), and define \[
D:=(D_1'+D_2'+D_3')+D_0+(D_1+\cdots+D_{m+1}).\] 
Let $\pi\colon \tilde{S}_{m,n}\to S_{m,n}$ denote the contraction of $D$. Then $S_{m,n}$ is a $\mathbb{Q}$-homology $\mathbb{CP}^2$ with a unique cyclic singularity of type $(p_{m,n},p_{m,n}-q_{m,n})$, and $\tilde{S}_{m,n}$ is its minimal resolution. Let $W_{m,n}$ denote the complement in $S_{m,n}$ of the cone neighborhood of the singularity. Our task reduces to computing $\pi_1(W_{m,n})$, or equivalently, $\pi_1(S_{m,n}^0)$, where $S_{m,n}^0=S_{m,n}\setminus\textup{Sing}(S_{m,n})=\tilde{S}_{m,n}\setminus D$ is the smooth locus of $S_{m,n}$. 

It follows from the construction that $S_{m,n}^0$ contains $\Sigma_n\setminus (\textup{section}+\textup{2 fibers})\cong \mathbb{C}^*\times \mathbb{C}$ as a Zariski open subset, where $\mathbb{C}^*:=\mathbb{C}\setminus\{0\}$. Therefore, the homomorphism $\pi_1(\mathbb{C}^*\times \mathbb{C})\to \pi_1(S_{m,n}^0)$ induced by inclusion is surjective by a transversality argument. Thus, $\pi_1(S_{m,n}^0)$ is a quotient of $\pi_1(\mathbb{C}^*\times \mathbb{C})\cong \mathbb{Z}$, and, in particular, $\pi_1(S_{m,n}^0)$ is abelian. Hence we only need to compute $H_1(S_{m,n}^0;\mathbb{Z})$. Note that there is a short exact sequence \cite[Lemma 2(2)]{Miyanishi-Zhang-1988}: \[
    0\to H_2(D;\mathbb{Z})\to \textup{Pic}(\tilde{S}_{m,n})\to H^2(S_{m,n}^0;\mathbb{Z})\to 0.
\]
Let $\Phi\colon \tilde{S}_{m,n}\to \mathbb{CP}^1$ denote the vertical $\mathbb{CP}^1$-fibration given in the configuration. Let $F$ be a fiber of $\Phi$. By keeping track of the blow-up process, one sees that the following linear equivalences hold: \[
F\sim D_3'+D_1'+2D_2'+2E_1 \sim D_{m+1}+D_1+2D_m+3D_{m-1}+\cdots+mD_2+mE_2.
\]
Note that $\textup{Pic}(\tilde{S}_{m,n})$ is a free abelian group of rank $m+6$ with a basis: \[
\{D_0, F, D_1',D_2',E_1,D_1,\dots,D_m,E_2 \}.
\]
In the quotient $\textup{Pic}(\tilde{S}_{m,n})/\langle D_0,D_1,\dots,D_{m+1},D_1',D_2',D_3'\rangle$, we have the relation $F=2E_1=mE_2$. Therefore, \[
H^2(S^0_{m,n};\mathbb{Z})=\begin{cases}
    \mathbb{Z} & \textrm{if $m$ is odd}, \\
    \mathbb{Z}\oplus \mathbb{Z}_2 & \textrm{if $m$ is even}. \end{cases}
\]
It follows that  \[
\pi_1(S^0_{m,n})\cong H_1(S^0_{m,n};\mathbb{Z})=\begin{cases}
    0 & \textrm{if $m$ is odd}, \\
    \mathbb{Z}_2 & \textrm{if $m$ is even}.
\end{cases}
\]  
In particular, $W_{m,n}$ is simply-connected when $m$ is odd, as desired. 
\end{proof}

\begin{remark} Observe that $p_{m,n}$ is odd when $m$ is odd. The simple connectivity of $W_{m,n}$ for odd $m$ can also be established using the topological argument provided in the proof of \cite[Theorem 3]{Lee-Park-2007} (see the proof of Proposition \ref{prop:Ballinger_construction}). \end{remark}

Next, we analyze the lattice embeddings and apply the changemaker criterion (Theorem \ref{thm:changemaker}) to show that $L(p_{m,n},q_{m,n})\notin \LK$ under a mild condition on $(m,n)$, thereby proving Theorem \ref{thm:filling2}.

\begin{proposition}\label{prop:changemaker_application} For each $m\geq 8$ and $n\geq 7$ with $m$ odd, the lens space $L(p_{m,n},q_{m,n})$ does not bound a compact, oriented, smooth $4$-manifold with $b_2=b_2^+=1$ built from a single $0$- and $2$-handle.
\end{proposition}

\begin{proof} We show that the lattice $Q_{X(p_{m,n},q_{m,n})}$ cannot embed as the orthogonal complement of a changemaker with square $-p_{m,n}$ in $-\mathbb{Z}^{b_2(X(p_{m,n},q_{m,n}))+1}=-\mathbb{Z}^{m+n}$ (Theorem \ref{thm:changemaker}). Suppose that $\iota\colon Q_{X(p_{m,n},q_{m,n})} \hookrightarrow -\mathbb{Z}^{m+n}$ is an embedding. Let $w_1,u_1,\dots,u_{n-1},w_2,v_1,\dots,v_{m-3},w_3 \in H_2(X(p_{m,n},q_{m,n});\mathbb{Z})$ denote the homology classes of the spheres corresponding to each vertex in the plumbing graph, as illustrated in Figure \ref{fig:plumbing_graph_of_X(p_m,n,q_m,n)2}.

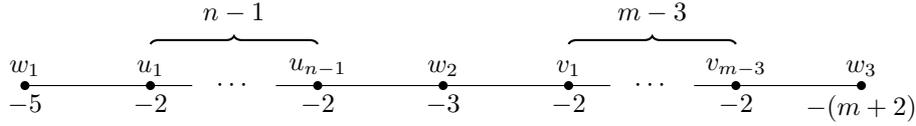
\begin{figure}[h]
\centering
\begin{tikzpicture}[scale=1.1]
\draw (-1,0) node[circle, fill, inner sep=1.2pt, black]{};
\draw (1,0) node[circle, fill, inner sep=1.2pt, black]{};
\draw (2.5,0) node[circle, fill, inner sep=1.2pt, black]{};
\draw (4,0) node[circle, fill, inner sep=1.2pt, black]{};
\draw (6,0) node[circle, fill, inner sep=1.2pt, black]{};
\draw (-2.5,0) node[circle, fill, inner sep=1.2pt, black]{};
\draw (7.5,0) node[circle, fill, inner sep=1.2pt, black]{};

\draw (-1,0) node[below]{$-2$};
\draw (1,0) node[below]{$-2$};
\draw (2.5,0) node[below]{$-3$};
\draw (4,0) node[below]{$-2$};
\draw (6,0) node[below]{$-2$};
\draw (-2.5,0) node[below]{$-5$};
\draw (7.5,0) node[below]{$-(m+2)$};

\draw (-1,0) node[above]{$u_1$};
\draw (1,0) node[above]{$u_{n-1}$};
\draw (2.5,0) node[above]{$w_2$};
\draw (4,0) node[above]{$v_1$};
\draw (6,0) node[above]{$v_{m-3}$};
\draw (-2.5,0) node[above]{$w_1$};
\draw (7.5,0) node[above]{$w_3$};

\draw (0,0) node{$\cdots$};
\draw (5,0) node{$\cdots$};

\draw (-2.5,0)--(-0.5,0) (0.5,0)--(1,0)  (1,0)--(4.5,0) (5.5,0)--(7.5,0) ;
\draw [thick,decorate,decoration={brace,amplitude=3pt},xshift=0pt]
	(-1,0.5) -- (1,0.5) node [black,midway,xshift=0pt,yshift=12pt] 
	{$n-1$};
 \draw [thick,decorate,decoration={brace,amplitude=3pt},xshift=0pt]
	(4,0.5) -- (6,0.5) node [black,midway,xshift=0pt,yshift=12pt] 
	{$m-3$};
\end{tikzpicture}
\caption{The plumbing graph of $X(p_{m,n},q_{m,n})$.}
\label{fig:plumbing_graph_of_X(p_m,n,q_m,n)2}
\end{figure}

Since $\iota(u_{1})^2=-2$,  we have $\iota(u_{1})=\pm e_i\pm e_j$ for some $1\leq i<j\leq m+n$. After a change of basis (i.e., an automorphism in $\textup{Aut}(-\mathbb{Z}^{m+n})=\textup{GL}(m+n,\mathbb{Z})\cap \textup{O}(m+n)$), we may assume $\iota(u_{1})=e_1-e_2$. Next, since $\iota(u_2)^2=-2$, we also have $\iota(u_{2})=\pm e_i\pm e_j$ for some $1\leq i<j\leq m+n$. Given that $\iota(u_1)\cdot \iota(u_2)=1$, it follows that $|\{1,2\}\cap \{i,j\}|=1$, implying that $i\leq 2$ and $j>2$. Thus, after another change of basis, we may assume $\iota(u_2)=e_2-e_3$. Next, write $\iota(u_3)=\pm e_i\pm e_j$ with $1\leq i<j\leq m+n$. Since $\iota(u_2)\cdot \iota(u_3)=1$, we must have $|\{2,3\}\cap \{i,j\}|=1$. If $\{2,3\}\cap \{i,j\}=\{2\}$, then $\iota(u_1)\cdot \iota(u_3)=0$ implies $\iota(u_3)=-e_1-e_2$. However, this leads to a contradiction since: \[
1=\iota(w_1)\cdot \iota(u_1)=\iota(w_1)\cdot (e_1-e_2)\equiv \iota(w_1)\cdot (-e_1-e_2) = \iota(w_1)\cdot \iota(u_3)=0\mod 2,\]
Thus, we must have $\{2,3\}\cap \{i,j\}=\{3\}$, meaning $i=3$ and $\iota(u_3)=e_3\pm e_j$. After another basis change, we may assume that $\iota(u_3)=e_3-e_4$. Continuing this process, we can assume that $\iota(u_\ell)=e_\ell-e_{\ell+1}$ for each $\ell=1,\dots,n-1$. It follows that $\iota(v_i)\cdot e_j=0$ for all $i=1,\dots,m-3$ and $j=1,\dots,n$. Using a similar argument as above, we can assume $\iota(v_\ell)=e_{n+\ell}-e_{n+1+\ell}$ for each $\ell=1,\dots,m-3$. 

Next, consider $\iota(w_2)$. Write $\iota(w_2)=\sum_{i=1}^{m+n} a_ie_i$ with $\sum_ia_i^2=3$. Considering $w_2\cdot u_\ell$'s and $w_2\cdot v_\ell$'s, we obtain $a_1=\cdots=a_{n-1}=a_n-1$ and $a_{n+1}+1=a_{n+2}=\cdots=a_{n+m-2}$. Therefore, if $n\geq 5$ and $m\geq 7$, we must have $a_1=\cdots=a_{n-1}=0$ and $a_{n+2}=\cdots=a_{n+m-2}=0$. Thus, $\iota(w_2)$ takes the form: $\iota(w_2)=e_n-e_{n+1}+a_{m+n-1}e_{m+n-1}+a_{m+n}e_{m+n}$ with $a_{m+n-1}^2+a_{m+n}^2=1$. After a change of basis, we may assume that $\iota(w_2)=e_n-e_{n+1}+e_{m+n-1}$.

Now, write $\iota(w_1)=\sum_{i=1}^{m+n} b_ie_i$ with $\sum_i b_i^2=5$. From the given intersection relations, we have $b_1+1=b_2=\cdots=b_n$ and $b_{n+1}=\cdots=b_{n+m-2}$. Thus, if $n\geq 7$ and $m\geq 8$, we must have $b_2=\cdots=b_n=0$ and $b_{n+1}=\cdots=b_{n+m-2}=0$. Therefore, $\iota(w_1)=-e_1+b_{n+m-1}e_{n+m-1}+b_{n+m}e_{n+m}$ with the condition $b_{n+m-1}^2+b_{n+m}^2=4$. Since $w_1\cdot w_2=0$, it follows that $b_{n+m-1}=0$ and $b_{n+m}=\pm 2$. We may assume that $\iota(w_1)=-e_1+2e_{n+m}$.

Finally, write $\iota(w_3)=\sum_{i=1}^{m+n}c_ie_i$ with $\sum_i c_i^2=m+2$. The coefficients $c_i$ must satisfy the following relations: 
$a:=c_1=\cdots=c_n, b:=c_{n+1}=\cdots=c_{n+m-3}, c_{n+m-2}=b+1, c_{n+m-1}=b-a,\text{ and } d:=c_{n+m}=a/2.$ 
Thus, we have \[
m+2=\sum_{i=1}^{m+n}c_i^2 = 4nd^2+(m-3)b^2+(b+1)^2+(b-2d)^2+d^2 \]
Since the right-hand side is $\geq (m-3)b^2\geq 4(m-3)>m+2$ for $m\geq 5$ and $|b|\geq 2$, we must have $b=\pm 1$ or $b=0$. If $b=-1$, then $(4n+5)d^2+4d=4$, which is clearly impossible. If $b=1$, then $d=0$ and we have: \[
\iota(w_3)=(e_{n+1}+\cdots+e_{n+m-3})+2e_{n+m-2}+e_{n+m-1}.\]
This gives an embedding $Q_{X(p_{m,n},q_{m,n})}\hookrightarrow -\mathbb{Z}^{m+n}$ as depicted in Figure \ref{fig:embedding_of_X(p_m,n,q_m,n)2} (a). Assuming that $m$ is odd, the orthogonal complement is generated by the vector:\[
2m(e_1+\cdots+e_n)+2(e_{n+1}+\cdots+e_{n+m-2})-2(m-1)e_{n+m-1}+me_{n+m}.\] 
This vector cannot be mapped to a changemaker under any automorphism of $-\mathbb{Z}^{m+n}$. 

If $b=0$, then $(4n+5)d^2=m+1$. Assuming that such a nonzero integer $d$ exists, we have: \[
\iota(w_3)=2d(e_1+\cdots+e_n)+e_{n+m-2}-2de_{n+m-1}+de_{n+m},
\]
and this gives an embedding $Q_{X(p_{m,n},q_{m,n})}\hookrightarrow -\mathbb{Z}^{m+n}$ as depicted in Figure \ref{fig:embedding_of_X(p_m,n,q_m,n)2} (b). 

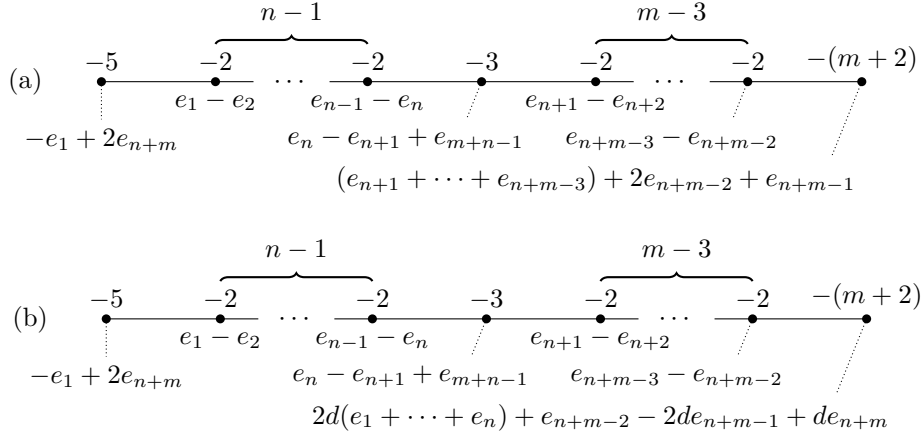
\begin{figure}[ht]
\centering
\begin{tikzpicture}[scale=1]
\node at (-3.5, 0) {(a)};
\draw (-1,0) node[circle, fill, inner sep=1.2pt, black]{};
\draw (1,0) node[circle, fill, inner sep=1.2pt, black]{};
\draw (2.5,0) node[circle, fill, inner sep=1.2pt, black]{};
\draw (4,0) node[circle, fill, inner sep=1.2pt, black]{};
\draw (6,0) node[circle, fill, inner sep=1.2pt, black]{};
\draw (-2.5,0) node[circle, fill, inner sep=1.2pt, black]{};
\draw (7.5,0) node[circle, fill, inner sep=1.2pt, black]{};

\draw (-1,0) node[above]{$-2$};
\draw (1,0) node[above]{$-2$};
\draw (2.5,0) node[above]{$-3$};
\draw (4,0) node[above]{$-2$};
\draw (6,0) node[above]{$-2$};
\draw (-2.5,0) node[above]{$-5$};
\draw (7.5,0) node[above]{$-(m+2)$};

\draw (-2.5,-0.5) node[below]{$-e_1+2e_{n+m}$};
\draw (-1,0) node[below]{$e_1-e_2$};
\draw (1,0) node[below]{$e_{n-1}-e_n$};
\draw (1.5,-0.5) node[below]{$e_n-e_{n+1}+e_{m+n-1}$};
\draw (4,0) node[below]{$e_{n+1}-e_{n+2}$};
\draw (5,-0.5) node[below]{$e_{n+m-3}-e_{n+m-2}$};
\draw (4,-1) node[below]{$(e_{n+1}+\cdots+e_{n+m-3})+2e_{n+m-2}+e_{n+m-1}$};

\draw (0,0) node{$\cdots$};
\draw (5,0) node{$\cdots$};

\draw [thick,decorate,decoration={brace,amplitude=3pt},xshift=0pt]
	(-1,0.5) -- (1,0.5) node [black,midway,xshift=0pt,yshift=12pt] 
	{$n-1$};
 \draw [thick,decorate,decoration={brace,amplitude=3pt},xshift=0pt]
	(4,0.5) -- (6,0.5) node [black,midway,xshift=0pt,yshift=12pt] 
	{$m-3$};

\draw[densely dotted] (-2.5,0)--(-2.5,-0.5) (2.5,0)--(2.3,-0.5) (6,0)--(5.8,-0.5) (7.5,0)--(7.1,-1);

\draw (-2.5,0)--(-0.5,0) (0.5,0)--(1,0)  (1,0)--(4.5,0) (5.5,0)--(7.5,0) ;
\end{tikzpicture}
\vspace{10pt}

\begin{tikzpicture}[scale=1]
\node at (-3.5, 0) {(b)};
\draw (-1,0) node[circle, fill, inner sep=1.2pt, black]{};
\draw (1,0) node[circle, fill, inner sep=1.2pt, black]{};
\draw (2.5,0) node[circle, fill, inner sep=1.2pt, black]{};
\draw (4,0) node[circle, fill, inner sep=1.2pt, black]{};
\draw (6,0) node[circle, fill, inner sep=1.2pt, black]{};
\draw (-2.5,0) node[circle, fill, inner sep=1.2pt, black]{};
\draw (7.5,0) node[circle, fill, inner sep=1.2pt, black]{};

\draw (-1,0) node[above]{$-2$};
\draw (1,0) node[above]{$-2$};
\draw (2.5,0) node[above]{$-3$};
\draw (4,0) node[above]{$-2$};
\draw (6,0) node[above]{$-2$};
\draw (-2.5,0) node[above]{$-5$};
\draw (7.5,0) node[above]{$-(m+2)$};

\draw (-2.5,-0.5) node[below]{$-e_1+2e_{n+m}$};
\draw (-1,0) node[below]{$e_1-e_2$};
\draw (1,0) node[below]{$e_{n-1}-e_n$};
\draw (1.5,-0.5) node[below]{$e_n-e_{n+1}+e_{m+n-1}$};
\draw (4,0) node[below]{$e_{n+1}-e_{n+2}$};
\draw (5,-0.5) node[below]{$e_{n+m-3}-e_{n+m-2}$};
\draw (4,-1) node[below]{$2d(e_1+\cdots+e_n)+e_{n+m-2}-2de_{n+m-1}+de_{n+m}$};

\draw (0,0) node{$\cdots$};
\draw (5,0) node{$\cdots$};

\draw [thick,decorate,decoration={brace,amplitude=3pt},xshift=0pt]
	(-1,0.5) -- (1,0.5) node [black,midway,xshift=0pt,yshift=12pt] 
	{$n-1$};
 \draw [thick,decorate,decoration={brace,amplitude=3pt},xshift=0pt]
	(4,0.5) -- (6,0.5) node [black,midway,xshift=0pt,yshift=12pt] 
	{$m-3$};

\draw[densely dotted] (-2.5,0)--(-2.5,-0.5) (2.5,0)--(2.3,-0.5) (6,0)--(5.8,-0.5) (7.5,0)--(7.1,-1);

\draw (-2.5,0)--(-0.5,0) (0.5,0)--(1,0)  (1,0)--(4.5,0) (5.5,0)--(7.5,0) ;
\end{tikzpicture}
\caption{Embeddings of $Q_{X(p_{m,n},q_{m,n})}$ into $-\mathbb{Z}^{m+n}$; (a) for general case, (b) for case where $(4n+5)d^2=m+1$ has an integer solution $d$. }
\label{fig:embedding_of_X(p_m,n,q_m,n)2}
\end{figure}

Observe that the vector 
$\mathbf{v}:=2(2d-1)(e_1+\cdots+e_n)+(4n+5)d(e_{n+1}+\cdots+e_{n+m-2})+((4n+1)d+2)e_{n+m-1}+(2d-1)e_{n+m}$
is contained in the orthogonal complement, and that $\mathbf{v}^2=-p_{m,n}$. Also note that $d\neq \pm 1$ as we are assuming that $m$ is odd. Therefore, every coefficient of $\mathbf{v}$ has absolute value $>1$, so $\mathbf{v}$ cannot be mapped to a changemaker under any automorphism of $-\mathbb{Z}^{m+n}$. We conclude that the lattice $Q_{X(p_{m,n},q_{m,n})}$ cannot embed in $-\mathbb{Z}^{m+n}$ as the orthogonal complement of a changemaker with square $-p_{m,n}$, even in the case where $(4n+5)d^2=m+1$ has an integer solution $d$. 
\end{proof}

\subsection*{$\mathbb{Q}$-homology $\mathbb{CP}^2$ construction for lens spaces in $\mathcal{L}_{Dehn}$}
It would be interesting to compare our construction of fillings with already known constructions. For integers $m\geq 2$ and $n\geq 1$, define two relatively prime integers $p_{m,n}$ and $q_{m,n}$ by
 \[
 p_{m,n}:=m^2n+m-1 \quad \textup{and} \quad q_{m,n}:=p_{m,n}-m^2=m^2n-m^2+m-1.
 \]
It is well-known by Moser \cite{Moser-1971} that $(ab\pm 1)$-surgery along the positive torus knot $T_{a,b}$ produces the lens space $L(ab\pm 1,-a^2)$. In particular, the lens space $L(p_{m,n},q_{m,n})$ can be obtained by $p_{m,n}$-surgery along the torus knot $T_{m,mn+1}$, and hence $L(p_{m,n},q_{m,n})\in \LK$. (In fact, by taking $k=m$ and $i=mn+1$, one sees that $(p_{m,n},q_{m,n})$ belongs to Berge's type I$_-$ \cite[Section 1.2]{Greene-2013}.)

On the other hand, a smooth $4$-manifold $W_{m,n}$ with $\pi_1=1$, $b_2=b_2^+=1$, and boundary $L(p_{m,n},q_{m,n})$ can be obtained from a $\mathbb{Q}$-homology $\mathbb{CP}^2$ construction, as follows. Note that \[
\frac{p_{m,n}}{p_{m,n}-q_{m,n}}=\left[ n+1, \left[2\right]^m,m\right].\]

Start with a configuration given by the union of the $(-n)$-section and a fiber in the Hirzebruch surface $\Sigma_n$ of degree $n$, as shown in Figure \ref{fig:example2} (a). By blowing up the marked intersection point twice, we obtain a configuration of rational curves depicted in Figure \ref{fig:example2} (b). Next, blow up the marked intersection point $m-2$ times, followed by a single blow-up at the final $(-1)$-curve, so that we have a configuration of rational curves shown in Figure \ref{fig:example2} (c). 

\begin{figure}[h]
    \centering
\begin{tikzpicture}[scale=0.8]
    \node at (0.5,-0.8) {(a)};
    \draw (-1.2,0)--(2,0) (1.5,-0.2)--(1.5,2);
    \draw (0,0) node[above]{$-n$};
    \draw (1.5,1) node[right]{$0$};
    \draw (1.5,0) node[circle, fill, inner sep=1.2pt, black]{};
\end{tikzpicture}\qquad
\begin{tikzpicture}[scale=0.9]
    \node at (6.5,-0.8) {(b)};
    \draw (4.5,0)--(8,0) (7.1,-0.2)--(7.7,1.2)  (7.7,1.8)--(7.1,3) (7.6,0.7)--(7.6,2.3);
    \draw (5.9,0) node[above]{$-(n+1)$};
    \draw (7.4,0.3) node[right]{$-2$};
    \draw (7.6,1.5) node[right]{$-1$};
    \draw (7.3,2.7) node[right]{$-2$};
    \draw (7.6,2) node[circle, fill, inner sep=1.2pt, black]{};
\end{tikzpicture}\qquad
\begin{tikzpicture}[scale=0.9]
    \node at (0.5,-0.8) {(c)};
        \draw (-2,0)--(3,0) (2.2,-0.2)--(2.9,1) (2.9,1.5)--(2.2,2.7) (2.2,2.4)--(2.9,3.6) (2.9,3.3)--(2.2,4.3) (2.45,3.1)--(3.5,2.5) ;
        \draw (2.7,1.4) node{$\vdots$};
        \draw (-1,0) node[above]{$-(n+1)$};
        \draw (2.5,0.4) node[left]{$-2$};
     \draw (2.5,2) node[left]{$-2$};
     \draw (2.5,2.9) node[left]{$-2$};
     \draw (3.5,2.5) node[right]{$-1$};
     \draw (2.3,4.2) node[right]{$-m$};
    \draw [decorate,decoration={brace,amplitude=3pt},xshift=0pt]
	(1.85,0.4) -- (1.85,2) node [black,midway,xshift=-18pt,yshift=0pt] 
	{$m-1$};
    \end{tikzpicture}
    \caption{Configurations over $m+1$ blow-ups from Hirzebruch surface $\Sigma_n$.}
    \label{fig:example2}
\end{figure}

Let $\tilde{S}_{m,n}$ denote the resulting surface, obtained by blowing up $\Sigma_n$ a total of $m+1$ times. Let $D$ denote the union of all rational curves given in Figure \ref{fig:example2} (c), except for the $(-1)$-curve. (Thus, $D$ consists of $m+2$ rational curves.) Let $\pi\colon \tilde{S}_{m,n}\to S_{m,n}$ denote the contraction of $D$. Then the surface $S_{m,n}$ is a $\mathbb{Q}$-homology $\mathbb{CP}^2$ with a unique cyclic singularity of type $(p_{m,n},p_{m,n}-q_{m,n})$, and $\tilde{S}_{m,n}$ is its minimal resolution. Let $W_{m,n}$ denote the complement of the cone neighborhood of the singularity in $S_{m,n}$. Then $W_{m,n}$ is a smooth $4$-manifold with $b_1=0$, $b_2=b_2^+=1$, and $\partial W_{m,n}=L(p_{m,n},q_{m,n})$. Moreover, observe that the smooth locus $S_{m,n}^0$ (which is homotopy equivalent to $W_{m,n}$) contains $\Sigma_n \setminus (\textup{section}+\textup{fiber}) \cong \mathbb{C}^2$ as a Zariski open subset. Hence, $W_{m,n}$ is simply-connected.

Thus, we have two smooth $4$-manifolds, $W_{m,n}$ and $W_{p_{m,n}}(T_{m,mn+1})$ (see Notation \ref{notation:knot_surgery}), both with $\pi_1=1$ and $b_2=b_2^+=1$, whose boundaries are the same $L(p_{m,n},q_{m,n})$. Determining whether the two $4$-manifolds, $W_{m,n}$ and $W_{p_{m,n}}(T_{m,mn+1})$, are diffeomorphic would be an intriguing problem.

\begin{question}\label{question:example} Are $W_{m,n}$ and $W_{p_{m,n}}(T_{m,mn+1})$ diffeomorphic?    
\end{question}

\smallskip

\section{Lens Spaces in $\LB\setminus \LS$}
\label{sec:filling1}
In this section, we explore lens spaces contained in $\LB\setminus \LS$, that is, those which bound a smooth 4-manifold with $b_1=0$ and $b_2=b_2^+=1$, but do not bound such a $4$-manifold under the additional constraint $\pi_1=1$. As noted in the introduction, $L(10,1)$ is an example of a lens space in $\LB\setminus \LS$. Here, we present an infinite family.

For integers $m,n\geq 0$, let \[
p_{m,n}:=(m+2)^2(m^2n+4mn+2m+4n+2)
\]
and \[
q_{m,n}:=m^4n+7m^3n+2m^3+17m^2n+8m^2+16mn+8m+4n+1.
\]
We have \[
\frac{p_{m,n}}{q_{m,n}}=\begin{cases}
    \left[ \left[ 2\right]^m, 2m+8, \left[2\right]^m \right] & \textrm{if $n=0$}, \\
    \left[\left[2\right]^m, m+5, \left[2\right]^{n-1},m+5, \left[2\right]^m \right] & \textrm{if $n\geq 1$}.
\end{cases}
\]

We first show that $L(p_{m,n},q_{m,n})\in \LB$, again by constructing appropriate $\mathbb{Q}$-homology $\mathbb{CP}^2$'s.

\begin{proposition}\label{prop:construction1} For each $m,n\geq 0$, the lens space $L(p_{m,n},q_{m,n})$ bounds a compact, oriented, smooth $4$-manifold $W_{m,n}$ such that $\pi_1(W_{m,n})\cong \mathbb{Z}_{m+2}$ and $b_2(W_{m,n})=b_2^+(W_{m,n})=1$.
\end{proposition}

\begin{proof} 
We show that there exists a $\mathbb{Q}$-homology $\mathbb{CP}^2$ having a unique cyclic singularity of type $(p_{m,n},p_{m,n}-q_{m,n})$ (Section \ref{subsec:qhcp2}). Note that \[
p_{m,n}-q_{m,n}=m^3n+7m^2n+2m^2+16mn+8m+12n+7
\]
and that \[
\frac{p_{m,n}}{p_{m,n}-q_{m,n}}=\left[m+2, \left[2\right]^{m+2}, n+2, \left[2\right]^{m+2}, m+2\right].
\]

Consider a configuration formed by the union of the $(-n)$-section and two fibers in the Hirzebruch surface $\Sigma_n$ of degree $n$, as shown in Figure \ref{fig:enter-label2} (a). By blowing up each of the two marked intersection points twice, we obtain the configuration of rational curves depicted in Figure \ref{fig:enter-label2} (b). And then, blow up each of the two marked intersection points $m$ times, followed by a single blow-up at each of the final two $(-1)$-curves. This process yields the configuration of rational curves given as in Figure \ref{fig:enter-label2} (c). 

\begin{figure}[h]
\centering
\begin{tikzpicture}[scale=0.8]
    \node at (-3.5, 1) {(a)};
    \draw (-2,0)--(2,0) (1.5,-0.2)--(1.5,2)  (-1.5,-0.2)--(-1.5,2) ;
    \draw (0,0) node[above]{$-n$};
    \draw (-1.5,1) node[left]{$0$};
    \draw (1.5,1) node[right]{$0$};
    \draw (1.5,0) node[circle, fill, inner sep=1.2pt, black]{};
    \draw (-1.5,0) node[circle, fill, inner sep=1.2pt, black]{};
\begin{scope}[shift={(2,0)}]
    \node at (3, 1) {(b)};
    \draw (4,0)--(8,0) (7.1,-0.2)--(7.7,1.2) (4.3,1.2)--(4.9,-0.2) (7.7,1.8)--(7.1,3) (4.3,1.8)--(4.9,3) (4.4,0.7)--(4.4,2.3) (7.6,0.7)--(7.6,2.3);
    \draw (5.9,0) node[below]{$-(n+2)$};
    \draw (4.6,0.3) node[left]{$-2$};
    \draw (7.4,0.3) node[right]{$-2$};
    \draw (4.4,1.5) node[left]{$-1$};
    \draw (7.6,1.5) node[right]{$-1$};
    \draw (4.7,2.7) node[left]{$-2$};
    \draw (7.3,2.7) node[right]{$-2$};
    \draw (7.6,2) node[circle, fill, inner sep=1.2pt, black]{};
    \draw (4.4,2) node[circle, fill, inner sep=1.2pt, black]{};
\end{scope}
\end{tikzpicture}\qquad

\vspace{20pt}
\begin{tikzpicture}[scale=1]
    \node at (-5, 2) {(c)};
    \draw (-3,0)--(3,0) (2.2,-0.2)--(2.9,1) (2.9,1.5)--(2.2,2.7) (2.2,2.4)--(2.9,3.6) (2.9,3.3)--(2.2,4.3) (-2.2,-0.2)--(-2.9,1) (-2.9,1.5)--(-2.2,2.7) (-2.2,2.4)--(-2.9,3.6) (-2.9,3.3)--(-2.2,4.3) (2.45,3.1)--(3.5,2.5) (-2.45,3.1)--(-3.5,2.5);
    \draw (2.7,1.4) node{$\vdots$};
    \draw (-2.7,1.4) node{$\vdots$};
    \draw (0,0) node[below]{$-(n+2)$};
    \draw (2.5,0.4) node[left]{$-2$};
    \draw (2.5,2) node[left]{$-2$};
    \draw (2.5,2.9) node[left]{$-2$};
    \draw (3.5,2.5) node[right]{$-1$};
    \draw (2.2,4.3) node[right]{$-(m+2)$};
    \draw (-2.5,0.4) node[right]{$-2$};
    \draw (-2.5,2) node[right]{$-2$};
    \draw (-2.5,3) node[right]{$-2$};
    \draw (-3.5,2.5) node[left]{$-1$};
    \draw (-2.3,4.3) node[left]{$-(m+2)$};
    \draw [decorate,decoration={brace,amplitude=3pt},xshift=0pt]
    (-1.85,2) -- (-1.85,0.4) node [black,midway,xshift=18pt,yshift=0pt] 
    {$m+1$};
    \draw (-2.2,4.3) node[right]{$D_1$};
    \draw (-2.75,3.4) node[right]{$D_2$};
    \draw (-2.65,2) node[left]{$D_3$};
    \draw (-2.6,0.4) node[left]{$D_{m+3}$};
    \draw (-3.1,2.9) node[left]{$E$};
    \draw (0,0) node[above]{$D_0$};
    \draw (2.2,4.3) node[left]{$D_1'$};
    \draw (2.75,3.4) node[left]{$D_2'$};
    \draw (2.6,2) node[right]{$D_3'$};
    \draw (2.6,0.4) node[right]{$D_{m+3}'$};
    \draw (3.1,2.9) node[right]{$E'$};
    \draw [decorate,decoration={brace,amplitude=3pt},xshift=0pt]
    (1.85,0.4) -- (1.85,2) node [black,midway,xshift=-18pt,yshift=0pt] 
    {$m+1$};
\end{tikzpicture}
\caption{Configurations over $2m+6$ blow-ups from Hirzebruch surface $\Sigma_n$.}
\label{fig:enter-label2}
\end{figure}

Let $\tilde{S}_{m,n}$ denote the resulting surface, obtained by performing $2m+6$ blow-ups from $\Sigma_n$. Label the rational curves as in Figure \ref{fig:enter-label2} (c), and define \[
D:=(D_1+\cdots+D_{m+3})+D_0+(D_1'+\cdots+D_{m+3}').\] 
Let $\pi\colon \tilde{S}_{m,n}\to S_{m,n}$ be the contraction of $D$. Then, $S_{m,n}$ is a $\mathbb{Q}$-homology $\mathbb{CP}^2$ with a unique cyclic singularity of type $(p_{m,n},p_{m,n}-q_{m,n})$, and $\tilde{S}_{m,n}$ is its minimal resolution. 

Let $W_{m,n}$ denote the complement in $S_{m,n}$ of the cone neighborhood of the singularity. It remains to show that $\pi_1(W_{m,n})\cong \mathbb{Z}_{m+2}$, or equivalently, that $\pi_1(S_{m,n}^0)\cong \mathbb{Z}_{m+2}$, where $S_{m,n}^0=S_{m,n}\setminus\textup{Sing}(S_{m,n})=\tilde{S}_{m,n}\setminus D$ is the smooth locus of $S_{m,n}$. As in the proof of Proposition \ref{prop:construction2}, we have $\pi_1(S_{m,n}^0)\cong H_1(S_{m,n}^0;\mathbb{Z})$.

Let $\Phi\colon \tilde{S}_{m,n}\to \mathbb{CP}^1$ denote the vertical $\mathbb{CP}^1$-fibration given in the configuration. Let $F$ be a fiber of $\Phi$. Then, we have the linear equivalences \begin{align*}
    F & \sim D_1+D_{m+3}+2D_{m+2}+3D_{m+1}+\cdots+(m+2)D_2+(m+2)E \\ & \sim D_1'+D_{m+3}'+2D_{m+2}'+3D_{m+1}'+\cdots+(m+2)D_2'+(m+2)E'.
\end{align*}
Note that $\textup{Pic}(\tilde{S}_{m,n})$ is a free abelian group of rank $2m+8$ with a basis: \[
\{D_0, F, D_2,\dots,D_{m+3},E, D_2',\dots,D_{m+3}',E' \}.
\]
In the quotient group $\textup{Pic}(\tilde{S}_{m,n})/\langle D_0,D_1,\dots,D_{m+3},D_1',\dots,D_{m+3}'\rangle$, we have the relation $F=(m+2)E=(m+2)E'$. Thus, we obtain $H^2({S}_{m,n}^0;\mathbb{Z})\cong \mathbb{Z}\oplus  \mathbb{Z}_{m+2}$ and it follows that $H_1(S_{m,n}^0;\mathbb{Z})\cong \mathbb{Z}_{m+2}$ as desired.    
\end{proof}

Next, we analyze the lattice embeddings of $Q_{X(p_{m,n},q_{m,n})}$ into the standard diagonal lattice of one higher rank to show that  $L(p_{m,n},q_{m,n})\notin \mathcal{L}_{H_1}$ and, consequently, $L(p_{m,n},q_{m,n})\notin \LS$. This completes the proof of Theorem \ref{thm:filling1}.

\begin{proposition}\label{prop:orthogonal_application} For each $m\geq 6$ and $n\geq 0$ with $n\neq 2,4,6$, the lens space $L(p_{m,n},q_{m,n})$ does not bound a compact, oriented, smooth $4$-manifold $W$ with $H_1(W;\mathbb{Z})=0$ and $b_2(W)=b_2^+(W)=1$.    
\end{proposition}

\begin{proof} 
We consider the following three cases and show that, for each case, the lattice embedding $\iota$ of $Q_{X(p_{m,n},q_{m,n})}$ into $-\mathbb{Z}^{2m+n+2}$ is essentially unique. However, the square of a vector generating the orthogonal complement of the image of $\iota$ does not match $-p_{m,n}$. Consequently, these lens spaces cannot be the boundary of a smooth 4-manifold with $H_1(W;\mathbb{Z})=0$ and $b_2(W)=b_2^+(W)=1$, by Corollary \ref{cor:orthogonal_complement}.

\textbf{Case 1:} We first consider the case $n=0$. Note that $p_{m,0}=2(m+1)(m+2)^2$. Suppose $\iota\colon Q_{X(p_{m,0},q_{m,0})}\hookrightarrow -\mathbb{Z}^{2m+2}$ is a lattice embedding. Let $u_1,\dots,u_m,w,v_1,\dots,v_m$ denote the homology classes in $H_2(X(p_{m,0},q_{m,0});\mathbb{Z})$ corresponding to the spheres in the plumbing graph, as shown in Figure \ref{fig:plumbing_graph_of_X(p_m,n,q_m,n)} (a). As in the proof of Proposition \ref{prop:changemaker_application}, we may assume that $\iota(u_\ell)=e_\ell-e_{\ell+1}$ for each $\ell=1,\dots,m$.

Next consider $\iota(v_1)$. Write $\iota(v_1)=\pm e_i\pm e_j$ with $i<j$. If $i\leq m+1$, then $\iota(v_1)\cdot \iota(u_k)\neq 0$ for some $k\in \{i-1,i,i+1\}$, which leads to a contradiction. Thus, we must have $i\geq m+2$. After a change of basis, we may assume that $\iota(v_1)=e_{m+2}-e_{m+3}$. Following the same reasoning as in the preceding paragraph, we can assume that $\iota(v_\ell)=e_{m+1+\ell}-e_{m+2+\ell}$ for each $\ell=1,\dots,m$.

Finally, consider $\iota(w)$. Write $\iota(w)=\sum_{i=1}^{2m+2} a_ie_i$. From the conditions $w\cdot u_\ell=0$ for $\ell<m$, $w\cdot v_\ell=0$ for $\ell>1$, and $w\cdot u_m=1=w\cdot v_1$, we derive the following relations:\[
a:=a_1=\cdots=a_m, \quad a_{m+1}=a+1,\quad b:=a_{m+3}=\cdots=a_{2m+2},\quad a_{m+2}=b-1.\] 
Additionally, we have: \[
2m+8=-\iota(w)^2=ma^2+(a+1)^2+(b-1)^2+mb^2=(a^2+b^2)m+(a+1)^2+(b-1)^2.\] 
This equation is impossible if $a^2+b^2\geq 4$, so we must have $|a|, |b|\leq 1$. It is straightforward to verify that the only solution is $(a,b)=(1,-1)$.

In conclusion, we have shown that there is a unique (up to an automorphism of $-\mathbb{Z}^{2m+2}$) embedding $Q_{X(p_{m,0},q_{m,0})}\hookrightarrow -\mathbb{Z}^{2m+2}$, as illustrated in Figure \ref{fig:embedding_X(p_m,n,q_m,n)} (a). The orthogonal complement is generated by the vector $e_1+\cdots+e_{2m+2}$, which has square $-(2m+2)\neq -p_{m,0}$.

\textbf{Case 2:} Next, we consider the case $n=1$. Note that $p_{m,1}=(m+2)^2(m^2+6m+6)$. Suppose $\iota\colon Q_{X(p_{m,1},q_{m,1})}\hookrightarrow -\mathbb{Z}^{2m+3}$ is a lattice embedding. Let $u_1,\dots,u_m,w_1,w_2,v_1,\dots,v_m\in H_2(X(p_{m,1},q_{m,1});\mathbb{Z})$ denote the homology classes of the spheres of the plumbing graph, as shown in Figure \ref{fig:plumbing_graph_of_X(p_m,n,q_m,n)} (b). As in Case 1, we may assume that $\iota(u_\ell)=e_\ell-e_{\ell+1}$ and $\iota(v_\ell)=e_{m+1+\ell}-e_{m+2+\ell}$ for each $\ell=1,\dots,m$. Next, write $\iota(w_1)=\sum_{i=1}^{2m+3} a_ie_i$.  From the given conditions, we derive: 
\[a:= a_1=\cdots=a_m,\quad a_{m+1}=a+1,\quad b:=a_{m+2}=\cdots=a_{2m+2}\]
Then, the following equation holds:\[
    m+5=ma^2+(a+1)^2+(m+1)b^2+c^2=(a^2+b^2)m+(a+1)^2+b^2+c^2\]
where $c:=a_{2m+3}$. Since $(a^2+b^2)m\geq 2m>m+5$ if $a^2+b^2\geq 2$ and $m\geq 6$, we must have $a^2+b^2\leq 1$. A case-by-case analysis shows that the only solutions are $(a,b,c)=(1,0,\pm 1)$ or $(a,b,c)=(0,0, \pm \sqrt{m+4})$ assuming $m+4$ is a perfect square. Thus, \[
\iota(w_1)=(e_1+\cdots+e_m)+2e_{m+1}\pm e_{2m+3} \text{ or } \iota(w_1)=e_{m+1}\pm \sqrt{m+4}e_{2m+3}.\] 
By a similar argument, we have \[
\iota(w_2)=-2e_{m+2}-(e_{m+3}+\cdots+e_{2m+2})\pm e_{2m+3}\text{ or } \iota(w_2)=-e_{m+3}\pm \sqrt{m+4}e_{2m+3}.\] 
Since $\iota(w_1)\cdot\iota(w_2)=1$, the only valid case is \[
    (\iota(w_1),\iota(w_2))=\left((e_1+\cdots+e_m)+2e_{m+1}\pm e_{2m+3}, -2e_{m+2}-(e_{m+3}+\cdots+e_{2m+2})\mp e_{2m+3} \right)
\]
Thus, for $m\geq 6$, there is a unique (up to an automorphism of $-\mathbb{Z}^{2m+3}$) embedding $Q_{X(p_{m,1},q_{m,1})}\hookrightarrow -\mathbb{Z}^{2m+3}$, as illustrated in Figure \ref{fig:embedding_X(p_m,n,q_m,n)} (b). The orthogonal complement is generated by the vector $e_1+\cdots+e_{2m+2}-(m+2)e_{2m+3}$, which has square $-(m^2+6m+6)\neq -p_{m,1}$.

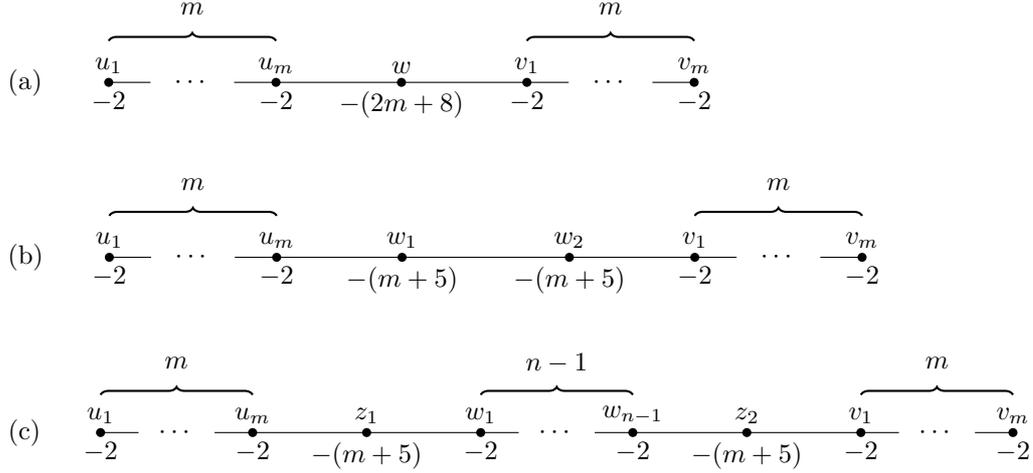
\begin{figure}[h]
\flushleft 
\begin{tikzpicture}[scale=1.1]
\node at (-2, 0) {(a)};
\draw (-1,0) node[circle, fill, inner sep=1.2pt, black]{};
\draw (1,0) node[circle, fill, inner sep=1.2pt, black]{};
\draw (2.5,0) node[circle, fill, inner sep=1.2pt, black]{};
\draw (4,0) node[circle, fill, inner sep=1.2pt, black]{};
\draw (6,0) node[circle, fill, inner sep=1.2pt, black]{};

\draw (-1,0) node[below]{$-2$};
\draw (1,0) node[below]{$-2$};
\draw (2.5,0) node[below]{$-(2m+8)$};
\draw (4,0) node[below]{$-2$};
\draw (6,0) node[below]{$-2$};

\draw (-1,0) node[above]{$u_1$};
\draw (1,0) node[above]{$u_m$};
\draw (2.5,0) node[above]{$w$};
\draw (4,0) node[above]{$v_1$};
\draw (6,0) node[above]{$v_m$};

\draw (0,0) node{$\cdots$};
\draw (5,0) node{$\cdots$};

\draw (-1,0)--(-0.5,0) (0.5,0)--(1,0)  (1,0)--(4.5,0) (5.5,0)--(6,0) ;
    \draw [thick,decorate,decoration={brace,amplitude=3pt},xshift=0pt]
	(-1,0.5) -- (1,0.5) node [black,midway,xshift=0pt,yshift=12pt] 
	{$m$};
 \draw [thick,decorate,decoration={brace,amplitude=3pt},xshift=0pt]
	(4,0.5) -- (6,0.5) node [black,midway,xshift=0pt,yshift=12pt] 
	{$m$};
\end{tikzpicture}
\vspace{15pt}

\begin{tikzpicture}[scale=1.1]
\node at (-2, 0) {(b)};
\draw (-1,0) node[circle, fill, inner sep=1.2pt, black]{};
\draw (1,0) node[circle, fill, inner sep=1.2pt, black]{};
\draw (2.5,0) node[circle, fill, inner sep=1.2pt, black]{};
\draw (4.5,0) node[circle, fill, inner sep=1.2pt, black]{};
\draw (8,0) node[circle, fill, inner sep=1.2pt, black]{};
\draw (6,0) node[circle, fill, inner sep=1.2pt, black]{};

\draw (-1,0) node[below]{$-2$};
\draw (1,0) node[below]{$-2$};
\draw (2.5,0) node[below]{$-(m+5)$};
\draw (4.5,0) node[below]{$-(m+5)$};
\draw (6,0) node[below]{$-2$};
\draw (8,0) node[below]{$-2$};

\draw (-1,0) node[above]{$u_1$};
\draw (1,0) node[above]{$u_m$};
\draw (4.5,0) node[above]{$w_2$};
\draw (2.5,0) node[above]{$w_1$};
\draw (6,0) node[above]{$v_1$};
\draw (8,0) node[above]{$v_m$};

\draw (0,0) node{$\cdots$};
\draw (7,0) node{$\cdots$};

\draw (-1,0)--(-0.5,0) (0.5,0)--(1,0)  (1,0)--(6.5,0) (7.5,0)--(8,0) ;
    \draw [thick,decorate,decoration={brace,amplitude=3pt},xshift=0pt]
	(-1,0.5) -- (1,0.5) node [black,midway,xshift=0pt,yshift=12pt] 
	{$m$};
 \draw [thick,decorate,decoration={brace,amplitude=3pt},xshift=0pt]
	(6,0.5) -- (8,0.5) node [black,midway,xshift=0pt,yshift=12pt] 
	{$m$};
\end{tikzpicture}
\vspace{15pt}

\begin{tikzpicture}[scale=1]
\node at (-2, 0) {(c)};
\draw (-1,0) node[circle, fill, inner sep=1.2pt, black]{};
\draw (1,0) node[circle, fill, inner sep=1.2pt, black]{};
\draw (2.5,0) node[circle, fill, inner sep=1.2pt, black]{};
\draw (4,0) node[circle, fill, inner sep=1.2pt, black]{};
\draw (6,0) node[circle, fill, inner sep=1.2pt, black]{};
\draw (7.5,0) node[circle, fill, inner sep=1.2pt, black]{};
\draw (9,0) node[circle, fill, inner sep=1.2pt, black]{};
\draw (11,0) node[circle, fill, inner sep=1.2pt, black]{};

\draw (-1,0) node[below]{$-2$};
\draw (1,0) node[below]{$-2$};
\draw (2.5,0) node[below]{$-(m+5)$};
\draw (4,0) node[below]{$-2$};
\draw (6,0) node[below]{$-2$};
\draw (7.5,0) node[below]{$-(m+5)$};
\draw (9,0) node[below]{$-2$};
\draw (11,0) node[below]{$-2$};

\draw (-1,0) node[above]{$u_1$};
\draw (1,0) node[above]{$u_m$};
\draw (2.5,0) node[above]{$z_1$};
\draw (4,0) node[above]{$w_1$};
\draw (6,0) node[above]{$w_{n-1}$};
\draw (7.5,0) node[above]{$z_2$};
\draw (9,0) node[above]{$v_1$};
\draw (11,0) node[above]{$v_m$};

\draw (0,0) node{$\cdots$};
\draw (5,0) node{$\cdots$};
\draw (10,0) node{$\cdots$};

\draw (-1,0)--(-0.5,0) (0.5,0)--(1,0)  (1,0)--(4.5,0) (5.5,0)--(9.5,0) (10.5,0)--(11,0) ;
\draw [thick,decorate,decoration={brace,amplitude=3pt},xshift=0pt]
	(-1,0.5) -- (1,0.5) node [black,midway,xshift=0pt,yshift=12pt] 
	{$m$};
 \draw [thick,decorate,decoration={brace,amplitude=3pt},xshift=0pt]
	(4,0.5) -- (6,0.5) node [black,midway,xshift=0pt,yshift=12pt] 
	{$n-1$};
 \draw [thick,decorate,decoration={brace,amplitude=3pt},xshift=0pt]
	(9,0.5) -- (11,0.5) node [black,midway,xshift=0pt,yshift=12pt] 
	{$m$};
\end{tikzpicture}
\caption{The plumbing graph of $X(p_{m,n},q_{m,n})$ for (a) $n=0$, (b) $n=1$, and (c) $n\geq 2$. }
\label{fig:plumbing_graph_of_X(p_m,n,q_m,n)}
\end{figure}

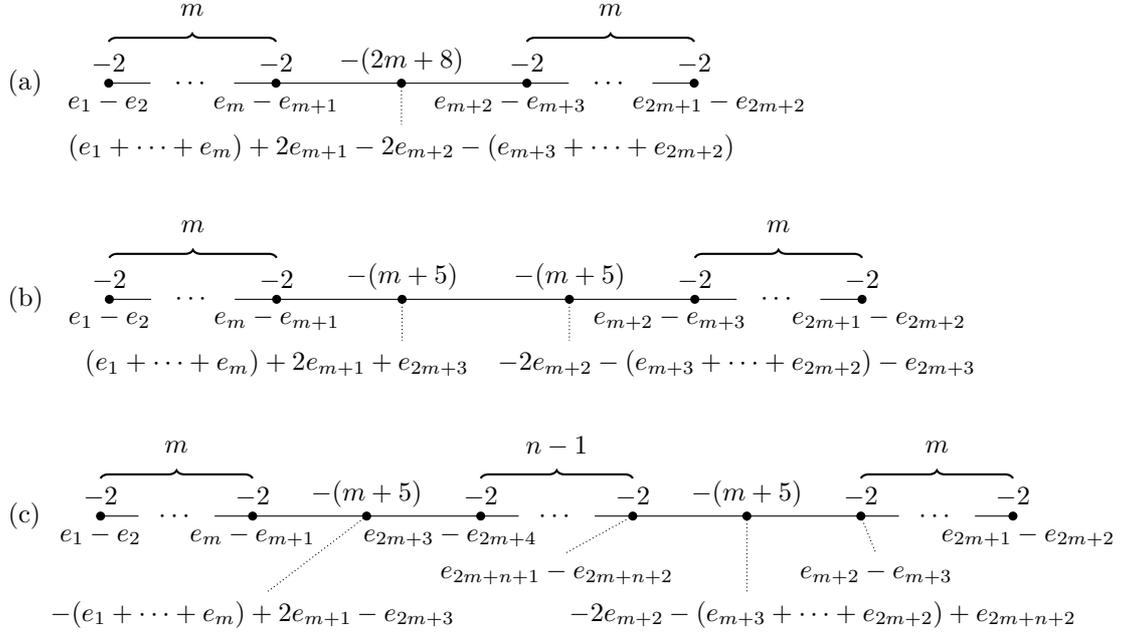
\begin{figure}[h]
\flushleft 
\begin{tikzpicture}[scale=1.1]
    \node at (-2, 0) {(a)};
    \draw (-1,0) node[circle, fill, inner sep=1.2pt, black]{};
    \draw (1,0) node[circle, fill, inner sep=1.2pt, black]{};
    \draw (2.5,0) node[circle, fill, inner sep=1.2pt, black]{};
    \draw (4,0) node[circle, fill, inner sep=1.2pt, black]{};
    \draw (6,0) node[circle, fill, inner sep=1.2pt, black]{};
    
    \draw (-1,0) node[above]{$-2$};
    \draw (1,0) node[above]{$-2$};
    \draw (2.5,0) node[above]{$-(2m+8)$};
    \draw (4,0) node[above]{$-2$};
    \draw (6,0) node[above]{$-2$};
    
    \draw (-1,0) node[below]{$e_1-e_2$};
    \draw (1,0) node[below]{$e_m-e_{m+1}$};
    \draw (2.5,-0.5) node[below]{$(e_1+\cdots+e_m)+2e_{m+1}-2e_{m+2}-(e_{m+3}+\cdots+e_{2m+2})$};
    \draw (3.8,0) node[below]{$e_{m+2}-e_{m+3}$};
    \draw (6.3,0) node[below]{$e_{2m+1}-e_{2m+2}$};
    
    \draw (0,0) node{$\cdots$};
    \draw (5,0) node{$\cdots$};
    \draw[densely dotted] (2.5,0)--(2.5,-0.5);
    \draw (-1,0)--(-0.5,0) (0.5,0)--(1,0)  (1,0)--(4.5,0) (5.5,0)--(6,0) ;
    \draw [thick,decorate,decoration={brace,amplitude=3pt},xshift=0pt]
	(-1,0.5) -- (1,0.5) node [black,midway,xshift=0pt,yshift=12pt] 
	{$m$};
 \draw [thick,decorate,decoration={brace,amplitude=3pt},xshift=0pt]
	(4,0.5) -- (6,0.5) node [black,midway,xshift=0pt,yshift=12pt] 
	{$m$};
\end{tikzpicture}
\vspace{15pt}

\begin{tikzpicture}[scale=1.1]
\node at (-2, 0) {(b)};
\draw (-1,0) node[circle, fill, inner sep=1.2pt, black]{};
\draw (1,0) node[circle, fill, inner sep=1.2pt, black]{};
\draw (2.5,0) node[circle, fill, inner sep=1.2pt, black]{};
\draw (4.5,0) node[circle, fill, inner sep=1.2pt, black]{};
\draw (8,0) node[circle, fill, inner sep=1.2pt, black]{};
\draw (6,0) node[circle, fill, inner sep=1.2pt, black]{};

\draw (-1,0) node[above]{$-2$};
\draw (1,0) node[above]{$-2$};
\draw (2.5,0) node[above]{$-(m+5)$};
\draw (4.5,0) node[above]{$-(m+5)$};
\draw (6,0) node[above]{$-2$};
\draw (8,0) node[above]{$-2$};

\draw (-1,0) node[below]{$e_1-e_2$};
\draw (1,0) node[below]{$e_m-e_{m+1}$};
\draw (1,-0.5) node[below]{$(e_1+\cdots+e_m)+2e_{m+1}+ e_{2m+3}$};
\draw (6.5,-0.5) node[below]{$-2e_{m+2}-(e_{m+3}+\cdots+e_{2m+2})- e_{2m+3}$};
\draw (5.7,0) node[below]{$e_{m+2}-e_{m+3}$};
\draw (8.2,0) node[below]{$e_{2m+1}-e_{2m+2}$};

\draw (0,0) node{$\cdots$};
\draw (7,0) node{$\cdots$};
\draw[densely dotted] (2.5,0)--(2.5,-0.5) (4.5,0)--(4.5,-0.5);
\draw (-1,0)--(-0.5,0) (0.5,0)--(1,0)  (1,0)--(6.5,0) (7.5,0)--(8,0) ;
\draw [thick,decorate,decoration={brace,amplitude=3pt},xshift=0pt]
	(-1,0.5) -- (1,0.5) node [black,midway,xshift=0pt,yshift=12pt] 
	{$m$};
 \draw [thick,decorate,decoration={brace,amplitude=3pt},xshift=0pt]
	(6,0.5) -- (8,0.5) node [black,midway,xshift=0pt,yshift=12pt] 
	{$m$};
\end{tikzpicture}
\vspace{15pt}

\begin{tikzpicture}[scale=1]
\node at (-2, 0) {(c)};
\draw (-1,0) node[circle, fill, inner sep=1.2pt, black]{};
\draw (1,0) node[circle, fill, inner sep=1.2pt, black]{};
\draw (2.5,0) node[circle, fill, inner sep=1.2pt, black]{};
\draw (4,0) node[circle, fill, inner sep=1.2pt, black]{};
\draw (6,0) node[circle, fill, inner sep=1.2pt, black]{};
\draw (7.5,0) node[circle, fill, inner sep=1.2pt, black]{};
\draw (9,0) node[circle, fill, inner sep=1.2pt, black]{};
\draw (11,0) node[circle, fill, inner sep=1.2pt, black]{};

\draw (-1,0) node[below]{$e_1-e_2$};
\draw (1,0) node[below]{$e_m-e_{m+1}$};
\draw (1,-1) node[below]{$-(e_1+\cdots+e_m)+2e_{m+1}-e_{2m+3}$};
\draw (3.6,0) node[below]{$e_{2m+3}-e_{2m+4}$};
\draw (5,-0.5) node[below]{$e_{2m+n+1}-e_{2m+n+2}$};
\draw (8.5,-1) node[below]{$-2e_{m+2}-(e_{m+3}+\cdots+e_{2m+2})+e_{2m+n+2}$};
\draw (9.2,-0.5) node[below]{$e_{m+2}-e_{m+3}$};
\draw (11.2,0) node[below]{$e_{2m+1}-e_{2m+2}$};

\draw (-1,0) node[above]{$-2$};
\draw (1,0) node[above]{$-2$};
\draw (2.5,0) node[above]{$-(m+5)$};
\draw (4,0) node[above]{$-2$};
\draw (6,0) node[above]{$-2$};
\draw (7.5,0) node[above]{$-(m+5)$};
\draw (9,0) node[above]{$-2$};
\draw (11,0) node[above]{$-2$};

\draw (0,0) node{$\cdots$};
\draw (5,0) node{$\cdots$};
\draw (10,0) node{$\cdots$};

\draw [thick,decorate,decoration={brace,amplitude=3pt},xshift=0pt]
	(-1,0.5) -- (1,0.5) node [black,midway,xshift=0pt,yshift=12pt] 
	{$m$};
 \draw [thick,decorate,decoration={brace,amplitude=3pt},xshift=0pt]
	(4,0.5) -- (6,0.5) node [black,midway,xshift=0pt,yshift=12pt] 
	{$n-1$};
 \draw [thick,decorate,decoration={brace,amplitude=3pt},xshift=0pt]
	(9,0.5) -- (11,0.5) node [black,midway,xshift=0pt,yshift=12pt] 
	{$m$};

\draw[densely dotted] (2.5,0)--(1.2,-1) (6,0)--(5.1,-0.5) (7.5,0)--(7.5,-1) (9,0)--(9.2,-0.5);

\draw (-1,0)--(-0.5,0) (0.5,0)--(1,0)  (1,0)--(4.5,0) (5.5,0)--(9.5,0) (10.5,0)--(11,0) ;
\end{tikzpicture}
\caption{An embedding of $Q_{X(p_{m,n},q_{m,n})}$ into $-\mathbb{Z}^{2m+n+2}$ for (a) $n=0$, (b) $n=1$, and (c) $n\geq 2$. }
\label{fig:embedding_X(p_m,n,q_m,n)} 
\end{figure}

\textbf{Case 3:} Finally, we consider the case $n\geq2$ with $n\neq 2,4,6$. Suppose $\iota\colon Q_{X(p_{m,n},q_{m,n})}\hookrightarrow -\mathbb{Z}^{2m+n+2}$ is a lattice embedding. Let $u_1,\dots,u_m,z_1,w_1,\dots,w_{n-1},z_2,v_1,\dots,v_m$ denote the homology classes in $H_2(X(p_{m,n},q_{m,n});\mathbb{Z})$ corresponding to the spheres in the plumbing graph, as shown in Figure \ref{fig:plumbing_graph_of_X(p_m,n,q_m,n)} (c). As in the previous cases, we may assume that $\iota(u_\ell)=e_\ell-e_{\ell+1}$ $(\ell=1,\dots,m)$, $\iota(v_\ell)=e_{m+1+\ell}-e_{m+2+\ell}$ $(\ell=1,\dots,m)$, and $\iota(w_\ell)=e_{2m+2+\ell}-e_{2m+3+\ell}$ $(\ell=1,\dots,n-1)$.

Next, consider $\iota(z_1)$. Write $\iota(z_1)=\sum_{i=1}^{2m+n+2} a_ie_i$. Then we have the following relations: 
\[a:= a_1=\cdots=a_m,\quad a_{m+1}=a+1,\quad b:=a_{m+2}=\cdots=a_{2m+2},\]
\[ \quad c:=a_{2m+4}=\cdots=a_{2m+n+2},\text{ and }a_{2m+3}=c-1, \]
Then we have
\[ m+5=ma^2+(a+1)^2+(m+1)b^2+(c-1)^2+(n-1)c^2.\]

Since $m\geq 6$, we must have $a^2+b^2\leq 1$. Now, assuming that $n\neq 2,4,6$, a case-by-case analysis shows that the only possibilities are $
(a,b,c)=(1,0,0)$ or $(a,b,c)=(0,0,c)$, where $c$ satisfies the equation $nc^2-2c=m+3$. Thus, $\iota(z_1)=(e_1+\cdots+e_m)+2e_{m+1}-e_{2m+3}$ or $\iota(z_1)=e_{m+1}+(c-1)e_{2m+3}+c(e_{2m+4}+\cdots+e_{2m+n+2})$ with $nc^2-2c=m+3$. Similarly, we have $\iota(z_2)=-2e_{m+2}-(e_{m+3}+\cdots+e_{2m+2})+e_{2m+n+2}$ or $\iota(z_2)=-e_{m+2}+c'(e_{2m+3}+\cdots+e_{2m+n+1})+(c'+1)e_{2m+n+2}$ with $nc'^2+2c'=m+3$. But we must have $\iota(z_1)\cdot \iota(z_2)=0$, so it follows that the only possibility is \[
\iota(z_1)=(e_1+\cdots+e_m)+2e_{m+1}-e_{2m+3}, \quad\iota(z_2)=-2e_{m+2}-(e_{m+3}+\cdots+e_{2m+2})+e_{2m+n+2}.
\]
We conclude that there is a unique (up to an automorphism of $-\mathbb{Z}^{2m+n+2}$) embedding   $Q_{X(p_{m,n},q_{m,n})}\hookrightarrow -\mathbb{Z}^{2m+n+2}$ as shown in Figure \ref{fig:embedding_X(p_m,n,q_m,n)} (c). The orthogonal complement is generated by the vector $(e_1+\cdots+e_{2m+2})+(m+2)(e_{2m+3}+\cdots+e_{2m+n+2})$, which has square $-((m+2)^2n+2m+2) \neq -p_{m,n}$.
\end{proof}

Another infinite family contained in $\LB \setminus \LS$ can also be found using a different construction and a distinct obstruction.

\begin{proposition}\label{prop:construction3}
For each $n>0$, the lens space $L(4n^2,2n-1)$ bounds a compact, oriented, smooth $4$-manifold with $b_1=0$ and $b_2=b_2^+=1$, but does not bound such a manifold with $H_1=0$.
\end{proposition}

\begin{proof}  
For each $n$, $L(4n^2,2n-1)$ bounds a rational homology 4-ball $V_n$ \cite{Lisca-2007}. Thus, the connected sum $V_n\# \mathbb{CP}^2$ is a smooth 4-manifold with $b_1=0$, $b_2=b_2^+=1$, and boundary $L(4n^2,2n-1)$. 

However, $L(4n^2,2n-1)$ does not bound a smooth 4-manifold with $b_2=b_2^+=1$ and $H_1=0$: Suppose that $L(4n^2,2n-1)$ bounds such a manifold $W_n$. Then $W_n$ must be spin \cite[Corollary 5.7.6]{Gompf-Stipsicz-1999}, as its intersection form is represented by the $1\times 1$ matrix $\left( 4n^2\right)$ by Corollary \ref{cor:linking_form}. Under the identification $\textup{Spin}^c(L(p,q))\cong \mathbb{Z}_p$ \cite{Ozsvath-Szabo-2003} (note that their orientation convention for lens spaces is opposite to ours), the two spin structures of $L(4n^2,2n-1)$ correspond to $n-1$ and $2n^2+n-1$. (In general, the spin structures of $L(p,q)$ correspond precisely to the integers among $(q-1)/2$ and $(p+q-1)/2$ \cite[p.~134]{Ue-2009}.) The corresponding $d$-invariants are $-\frac{2n+1}{4}$ and $\frac{2n-1}{4}$, respectively, which contradicts \cite[Corollary 2.16]{Jo-Park-Park-2025-01}. 
\end{proof}

\begin{remark}\label{rmk:prop4.3}
The argument in the proof of Proposition \ref{prop:construction3} shows that the collection of all lens spaces bounding a smooth rational homology 4-ball, denoted by $\mathcal{R}$ in \cite{Lisca-2007}, is a subcollection of $\LB$. On the other hand, the lens spaces in Proposition \ref{prop:orthogonal_application} do not bound rational homology balls in general: If a lens space $L(p,q)$ bounds a rational homology ball, then $p$ must be a square number \cite{Lisca-2007}. Thus, the examples in Proposition \ref{prop:orthogonal_application} represent relatively nontrivial elements of $\LB\setminus \LS$.
\end{remark}

\section{Discussion of Tange's and Ballinger's Examples}\label{appendix}
In this section, we discuss the examples of Tange \cite{Tange-2018} and Ballinger \cite{Ballinger-2022} that are expected to lie in $\LS \setminus \LK$.

\subsection{Tange's Examples}
 As mentioned in the introduction, $L(17,2)$ is an example of a lens space contained in $\LS \setminus \LK$. We note that a simply-connected, smooth 4-manifold with $b_2=b_2^+=1$ bounded by $L(17,2)$ can also be obtained from a $\mathbb{Q}$-homology $\mathbb{CP}^2$ as follows: Consider a configuration of rational curves shown in Figure \ref{fig:qhcp2_(17,2)}, where the dotted lines represent $(-1)$-curves, and the solid lines represent $(-2)$-curves, except for a unique $(-3)$-curve. This configuration is obtained by blowing up seven times from the union of the negative section and two fibers in the Hirzebruch surface $\Sigma_2$. Contracting the eight solid lines yields a $\mathbb{Q}$-homology $\mathbb{CP}^2$ with a unique cyclic singularity of type $(17,15)$. Applying the argument in the proof of Proposition \ref{prop:construction2}, it is straightforward to verify that its smooth locus is simply-connected.
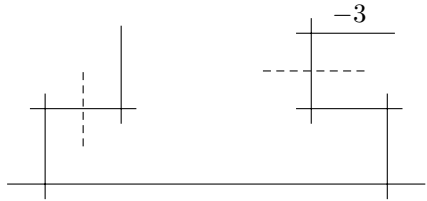
\begin{figure}[h]
\begin{tikzpicture}[scale=1]
      \draw (-3.5,0)--(2,0) (1.5,-0.2)--(1.5,1.2) (1.7,1)--(0.3,1) (0.5,0.8)--(0.5,2.2) (1.6,2)--(0.3,2) (-3,-0.2)--(-3,1.2) (-3.2,1)--(-1.8,1) (-2,0.8)--(-2,2.1);
       \draw[densely dashed] (1.2,1.5)--(-0.2,1.5) (-2.5,0.5)--(-2.5,1.5);
        \draw (1,2) node[above]{$-3$};
\end{tikzpicture}
\caption{A configuration after seven blow-ups from Hirzebruch surface $\Sigma_2$.}
\label{fig:qhcp2_(17,2)}
\end{figure}

In \cite{Tange-2018}, Tange studied a single 2-handle cobordism from Brieskorn spheres to lens spaces. Since certain Brieskorn spheres are known to bound contractible smooth 4-manifolds, this allows us to construct simply-connected, smooth 4-manifolds with $b_2=1$ and a lens space boundary by taking the union of the cobordism and the contractible 4-manifold along the Brieskorn sphere. Several infinite families of lens spaces contained in $\LS$, but expected not to be contained in $\LK$, are presented in \cite{Tange-2018}. Let us consider one of his families (similar arguments apply to the others), given by 
\[
L(35\ell^2+21\ell+3,21\ell^2+14\ell+2) \quad (\ell=1,2,\dots).
\]
Note that Tange's orientation convention of lens spaces is opposite to ours. We confirm that this family is indeed not contained in $\LK$. Let $p_\ell:=35\ell^2+21\ell+3$ and $q_\ell:=21\ell^2+14\ell+2$, so that 
\[
\frac{p_\ell}{q_\ell}=\left[2,3,\ell+1,8,\left[2\right]^{\ell-1} \right].
\]

\begin{proposition}\label{prop:changemaker_application_2} For each $\ell\geq 1$, the lens space $L(p_\ell,q_\ell)$ does not bound a compact, oriented, smooth $4$-manifold with $b_2=b_2^+=1$ built from a single $0$- and $2$-handle. Consequently, $L(p_\ell,q_\ell)$ cannot be obtained by a positive integer surgery along any knot in $S^3$. 
\end{proposition}

\begin{proof} 
Similar to the proof of Proposition \ref{prop:changemaker_application}, we show that the lattice $Q_{X(p_\ell,q_\ell)}$ does not embed as the orthogonal complement of a changemaker in $-\mathbb{Z}^{\ell+4}$. Assume that $\iota\colon Q_{X(p_\ell,q_\ell)} \hookrightarrow -\mathbb{Z}^{\ell+4}$ is an embedding. Let $u_1,u_2,u_3,u_4,v_1,\dots,v_{\ell-1} \in H_2(X(p_\ell,q_\ell);\mathbb{Z})$ denote the homology classes of the spheres corresponding to the vertices in the plumbing graph, as shown in Figure \ref{fig:plumbing_graph_of_X(p_l,q_l)}.

\begin{figure}[h]
\centering
\begin{tikzpicture}[scale=1.2]
\draw (-1.5,0) node[circle, fill, inner sep=1.2pt, black]{};
\draw (0.5,0) node[circle, fill, inner sep=1.2pt, black]{};
\draw (2.5,0) node[circle, fill, inner sep=1.2pt, black]{};
\draw (4.5,0) node[circle, fill, inner sep=1.2pt, black]{};
\draw (8,0) node[circle, fill, inner sep=1.2pt, black]{};
\draw (6,0) node[circle, fill, inner sep=1.2pt, black]{};

\draw (-1.5,0) node[below]{$-2$};
\draw (0.5,0) node[below]{$-3$};
\draw (2.5,0) node[below]{$-(\ell+1)$};
\draw (4.5,0) node[below]{$-8$};
\draw (6,0) node[below]{$-2$};
\draw (8,0) node[below]{$-2$};

\draw (-1.5,0) node[above]{$u_1$};
\draw (0.5,0) node[above]{$u_2$};
\draw (2.5,0) node[above]{$u_3$};
\draw (4.5,0) node[above]{$u_4$};

\draw (6,0) node[above]{$v_1$};
\draw (8,0) node[above]{$v_{\ell-1}$};

\draw (7,0) node{$\cdots$};
\draw [thick,decorate,decoration={brace,amplitude=3pt},xshift=0pt]
	(6,0.5) -- (8,0.5) node [black,midway,xshift=0pt,yshift=12pt] 
	{$\ell-1$};
\draw (-1.5,0)--(6.5,0) (7.5,0)--(8,0) ;
\end{tikzpicture}
\caption{The plumbing graph of $X(p_\ell,q_\ell)$ (Proposition \ref{prop:changemaker_application_2}).}
\label{fig:plumbing_graph_of_X(p_l,q_l)}
\end{figure}
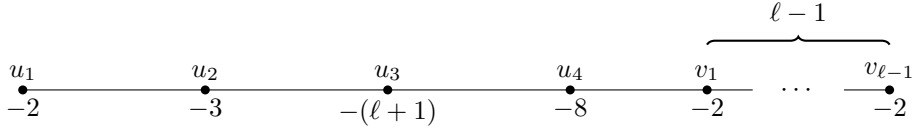

We first consider the case $\ell\geq 10$. As in the proof of Proposition \ref{prop:changemaker_application}, we may assume that $\iota(v_i)=e_i-e_{i+1}$ for $i=1,\dots,\ell-1$ and $\iota(u_1)=e_{\ell+1}-e_{\ell+2}$. Next, consider $\iota(u_2)$. We have $\iota(u_2)=\pm e_i \pm e_j \pm e_k$ with $i<j<k$. From the conditions $u_2\cdot v_t=0$ for $t=1,\dots,\ell-1$ and $u_2\cdot u_1=1$, it easily follows that either $i=\ell+1$ or $i=\ell+2$. By applying an automorphism of $-\mathbb{Z}^{\ell+4}$, we may assume without loss of generality that $(i,j,k)=(\ell+2,\ell+3,\ell+4)$ and that $\iota(u_2)=e_{\ell+2}+e_{\ell+3}+e_{\ell+4}$.

Now consider $\iota(u_4)$. Write $\iota(u_4)=\sum_{i=1}^{\ell+4}a_ie_i$. Then we have $a:=a_1+1=a_2=\cdots=a_\ell$, $b:=a_{\ell+1}=a_{\ell+2}$, $a_{\ell+2}+a_{\ell+3}+a_{\ell+4}=0$ and $\sum_i a_i^2 = 8$. Since $\sum_i a_i^2 \geq (\ell-1)a^2 \geq 9a^2$, we must have $a=0$. This implies $a_1=-1$ and $a_2=\cdots=a_\ell=0$. Letting $c:=a_{\ell+3}$ and $d:=a_{\ell+4}$, it follows that $b+c+d=0$ and $1+2b^2+c^2+d^2=8$. Thus, the possible solutions for $(b,c,d)$ are $(1,-2,1), (1,1,-2), (-1,2,-1)$, or $(-1,-1,2)$. Clearly, the first two cases are equivalent and the last two cases are also equivalent. Therefore, we may assume $(b,c,d)=(1,1,-2)$ or $(-1,-1,2)$.

\textbf{Case 1:} Assume that $(b,c,d)=(-1,-1,2)$, i.e., $\iota(u_4)=-e_1-e_{\ell+1}-e_{\ell+2}-e_{\ell+3}+2e_{\ell+4}$. Now, write $\iota(u_3)=\sum_{i=1}^{\ell+4}b_ie_i$ with $\sum_i b_i^2 = \ell+1$. We have $\alpha:=b_1=\cdots=b_\ell$, $\beta:=b_{\ell+1}=b_{\ell+2}$ and $\beta+\gamma+\delta=-1=-\alpha-2\beta-\gamma+2\delta$, where $\gamma:=b_{\ell+3}$ and $\delta:=b_{\ell+4}$. Since $\ell+1=\sum_i b_i^2\geq \ell \alpha^2$, we must have $\alpha=\pm 1$ or $\alpha=0$. If $\alpha=1$, then we have \[\beta+\gamma+\delta=-1,\quad -2\beta-\gamma+2\delta=0,\quad 2\beta^2+\gamma^2+\delta^2=1,\] but this has no solution. If $\alpha=-1$, then we have \[\beta+\gamma+\delta=-1,\quad -2\beta-\gamma+2\delta=-2, \quad 2\beta^2+\gamma^2+\delta^2=1.\] The unique solution is $(\beta,\gamma,\delta)=(0,0,-1)$. In this case, $\iota(u_3)=-(e_1+\cdots+e_\ell)-e_{\ell+4}$. This corresponds the embedding shown in Figure \ref{fig:embedding_X(p_l,q_l)3} (a). The orthogonal complement is generated by the vector \[
(e_1+\cdots+e_\ell)-(3\ell+1)(e_{\ell+1}+e_{\ell+2})+(4\ell+1)e_{\ell+3}-\ell e_{\ell+4}.
\]
However, this vector cannot be mapped to a changemaker under any automorphism of $-\mathbb{Z}^{\ell+4}$.

If $\alpha=0$, then we have \[\beta+\gamma+\delta=-1=-2\beta-\gamma+2\delta,\quad 2\beta^2+\gamma^2+\delta^2=\ell+1.\] From these equations, we find \[\beta=3\delta+2,\quad \gamma=-4\delta-3,\quad \ell=35\delta^2+48\delta+16.\] Thus, $\ell=35\delta^2+48\delta+16$ must have an integer solution $\delta$. In this case, $\iota(u_3)=(3\delta+2)(e_{\ell+1}+e_{\ell+2})-(4\delta+3)e_{\ell+3}+\delta e_{\ell+4}$. This corresponds to the embedding shown in Figure \ref{fig:embedding_X(p_l,q_l)3} (b). The orthogonal complement is generated by the vector \[
(35\delta+24)(e_1+\cdots+e_\ell)-(5\delta+3)(e_{\ell+1}+e_{\ell+2})-(5\delta+4)e_{\ell+3}+(10\delta+7)e_{\ell+4},
\]
which also cannot be mapped to a changemaker under any automorphism of $-\mathbb{Z}^{\ell+4}$.

\textbf{Case 2:} Assume that $(b,c,d)=(1,1,-2)$, i.e., $\iota(u_4)=-e_1+e_{\ell+1}+e_{\ell+2}+e_{\ell+3}-2e_{\ell+4}$. Now, write $\iota(u_3)=\sum_{i=1}^{\ell+4}b_ie_i$ with $\sum_i b_i^2 = \ell+1$. We have $\alpha:=b_1=\cdots=b_\ell$, $\beta:=b_{\ell+1}=b_{\ell+2}$ along with the conditions
\[\beta+\gamma+\delta=-1=-\alpha+2\beta+\gamma-2\delta,\] 
where $\gamma:=b_{\ell+3}$ and $\delta:=b_{\ell+4}$. As in Case 1, we must have $\alpha=\pm 1$ or $\alpha=0$. If $\alpha=1$, then we have
\[\beta+\gamma+\delta=-1,\quad 2\beta+\gamma-2\delta=0,\quad 2\beta^2+\gamma^2+\delta^2=1\] 
but this system has no common integer solutions. If $\alpha=-1$, it is also easily verified that no solutions exist. If $\alpha=0$, the equations reduce to \[\beta=3\delta,\quad \gamma=-4\delta-1, \quad\ell=35\delta^2+8\delta.\] Thus, $\ell=35\delta^2+8\delta$ must have an integer solution $\delta$. In this case, we have $\iota(u_3)=3\delta(e_{\ell+1}+e_{\ell+2})-(4\delta+1)e_{\ell+3}+\delta e_{\ell+4}$. This corresponds to the embedding shown in Figure \ref{fig:embedding_X(p_l,q_l)3} (c). The orthogonal complement is generated by the vector \[
-(35\delta+4)(e_1+\cdots+e_\ell)-(5\delta+1)(e_{\ell+1}+e_{\ell+2})-5\delta e_{\ell+3}+(10\delta+1)e_{\ell+4},
\]
which also cannot be mapped to a changemaker under any automorphism of $-\mathbb{Z}^{\ell+4}$.
This completes the proof for the case $\ell \geq 10$.

\begin{figure}[h]
\flushleft 
\begin{tikzpicture}[scale=.9]
    \node at (-3, 0) {(a)};
    \draw (-1.5,0) node[circle, fill, inner sep=1.2pt, black]{};
    \draw (0.5,0) node[circle, fill, inner sep=1.2pt, black]{};
    \draw (2.5,0) node[circle, fill, inner sep=1.2pt, black]{};
    \draw (4.5,0) node[circle, fill, inner sep=1.2pt, black]{};
    \draw (8,0) node[circle, fill, inner sep=1.2pt, black]{};
    \draw (6,0) node[circle, fill, inner sep=1.2pt, black]{};
    
    \draw (-1.5,0) node[above]{$-2$};
    \draw (0.5,0) node[above]{$-3$};
    \draw (2.5,0) node[above]{$-(\ell+1)$};
    \draw (4.5,0) node[above]{$-8$};
    \draw (6,0) node[above]{$-2$};
    \draw (8,0) node[above]{$-2$};
    
    \draw (-1.5,0) node[below]{$e_{\ell+1}-e_{\ell+2}$};
    \draw (0.7,-0.5) node[below]{$e_{\ell+2}+e_{\ell+3}+e_{\ell+4}$};
    \draw (2.5,-1) node[below]{$-(e_1+\cdots+e_\ell)-e_{\ell+4}$};
    \draw (5.5,-0.5) node[below]{$-e_1-e_{\ell+1}-e_{\ell+2}-e_{\ell+3}+2e_{\ell+4}$};
    \draw (6,0) node[below]{$e_1-e_2$};
    \draw (8,0) node[below]{$e_{\ell-1}-e_\ell$};
    
    \draw[densely dotted] (0.5,0)--(0.5,-0.5) (2.5,0)--(2.5,-1) (4.5,0)--(4.5,-0.5) ;
    \draw (7,0) node{$\cdots$};
    \draw [thick,decorate,decoration={brace,amplitude=3pt},xshift=0pt]
    	(6,0.5) -- (8,0.5) node [black,midway,xshift=0pt,yshift=12pt] 
    	{$\ell-1$};
    \draw (-1.5,0)--(6.5,0) (7.5,0)--(8,0);
\end{tikzpicture}\vspace{15pt}

\vspace{-1 em}

\begin{tikzpicture}[scale=1]
\node at (-3, 0) {(b)};
    \draw (-1.5,0) node[circle, fill, inner sep=1.2pt, black]{};
    \draw (0.5,0) node[circle, fill, inner sep=1.2pt, black]{};
    \draw (2.5,0) node[circle, fill, inner sep=1.2pt, black]{};
    \draw (4.5,0) node[circle, fill, inner sep=1.2pt, black]{};
    \draw (8,0) node[circle, fill, inner sep=1.2pt, black]{};
    \draw (6,0) node[circle, fill, inner sep=1.2pt, black]{};
    
    \draw (-1.5,0) node[above]{$-2$};
    \draw (0.5,0) node[above]{$-3$};
    \draw (2.5,0) node[above]{$-(\ell+1)$};
    \draw (4.5,0) node[above]{$-8$};
    \draw (6,0) node[above]{$-2$};
    \draw (8,0) node[above]{$-2$};
    
    \draw (-1.5,0) node[below]{$e_{\ell+1}-e_{\ell+2}$};
    \draw (0.7,-0.5) node[below]{$e_{\ell+2}+e_{\ell+3}+e_{\ell+4}$};
    \draw (2.5,-1) node[below]{$(3\delta+2)(e_{\ell+1}+e_{\ell+2})-(4\delta+3)e_{\ell+3}+\delta e_{\ell+4}$};
    \draw (5.5,-0.5) node[below]{$-e_1-e_{\ell+1}-e_{\ell+2}-e_{\ell+3}+2e_{\ell+4}$};
    \draw (6,0) node[below]{$e_1-e_2$};
    \draw (8,0) node[below]{$e_{\ell-1}-e_\ell$};
    
    \draw[densely dotted] (0.5,0)--(0.5,-0.5) (2.5,0)--(2.5,-1) (4.5,0)--(4.5,-0.5) ;
    \draw (7,0) node{$\cdots$};
    \draw [thick,decorate,decoration={brace,amplitude=3pt},xshift=0pt]
    	(6,0.5) -- (8,0.5) node [black,midway,xshift=0pt,yshift=12pt] 
    	{$\ell-1$};
    \draw (-1.5,0)--(6.5,0) (7.5,0)--(8,0) ;
\end{tikzpicture}\vspace{15pt}
\vspace{-1 em}
\begin{tikzpicture}[scale=1]
    \node at (-3, 0) {(c)};
    \draw (-1.5,0) node[circle, fill, inner sep=1.2pt, black]{};
    \draw (0.5,0) node[circle, fill, inner sep=1.2pt, black]{};
    \draw (2.5,0) node[circle, fill, inner sep=1.2pt, black]{};
    \draw (4.5,0) node[circle, fill, inner sep=1.2pt, black]{};
    \draw (8,0) node[circle, fill, inner sep=1.2pt, black]{};
    \draw (6,0) node[circle, fill, inner sep=1.2pt, black]{};
    
    \draw (-1.5,0) node[above]{$-2$};
    \draw (0.5,0) node[above]{$-3$};
    \draw (2.5,0) node[above]{$-(\ell+1)$};
    \draw (4.5,0) node[above]{$-8$};
    \draw (6,0) node[above]{$-2$};
    \draw (8,0) node[above]{$-2$};
    
    \draw (-1.5,0) node[below]{$e_{\ell+1}-e_{\ell+2}$};
    \draw (0.7,-0.5) node[below]{$e_{\ell+2}+e_{\ell+3}+e_{\ell+4}$};
    \draw (2.5,-1) node[below]{$3\delta(e_{\ell+1}+e_{\ell+2})-(4\delta+1)e_{\ell+3}+\delta e_{\ell+4}$};
    \draw (5.5,-0.5) node[below]{$-e_1+e_{\ell+1}+e_{\ell+2}+e_{\ell+3}-2e_{\ell+4}$};
    \draw (6,0) node[below]{$e_1-e_2$};
    \draw (8,0) node[below]{$e_{\ell-1}-e_\ell$};
    
    \draw[densely dotted] (0.5,0)--(0.5,-0.5) (2.5,0)--(2.5,-1) (4.5,0)--(4.5,-0.5) ;
    \draw (7,0) node{$\cdots$};
    \draw [thick,decorate,decoration={brace,amplitude=3pt},xshift=0pt]
    	(6,0.5) -- (8,0.5) node [black,midway,xshift=0pt,yshift=12pt] 
    	{$\ell-1$};
    \draw (-1.5,0)--(6.5,0) (7.5,0)--(8,0) ;
\end{tikzpicture}\vspace{15pt}

\begin{tikzpicture}[scale=1]
    \node at (-3,0) {(d)};
    \draw (-1.5,0) node[circle, fill, inner sep=1.2pt, black]{};
    \draw (0.5,0) node[circle, fill, inner sep=1.2pt, black]{};
    \draw (2.5,0) node[circle, fill, inner sep=1.2pt, black]{};
    \draw (4.5,0) node[circle, fill, inner sep=1.2pt, black]{};
    \draw (6.5,0) node[circle, fill, inner sep=1.2pt, black]{};
    
    \draw (-1.5,0) node[above]{$-2$};
    \draw (0.5,0) node[above]{$-3$};
    \draw (2.5,0) node[above]{$-3$};
    \draw (4.5,0) node[above]{$-8$};
    \draw (6.5,0) node[above]{$-2$};
    
    \draw (-1.5,0) node[below]{$e_3-e_4$};
    \draw (0.5,-0.5) node[below]{$e_1-e_2-e_3$};
    \draw (2.5,0) node[below]{$e_3+e_4+e_5$};
    \draw (4.5,-0.5) node[below]{$-e_1-e_3-e_4+e_5+2e_6$};
    \draw (6.5,0) node[below]{$e_1+e_2$};
    
    \draw (-1.5,0)--(6.5,0) ;
    \draw[densely dotted] (0.5,0)--(0.5,-0.5) (4.5,0)--(4.5,-0.5);
\end{tikzpicture}
\caption{Embeddings of $Q_{X(p_\ell,q_\ell)}$ into $-\mathbb{Z}^{\ell+4}$, (b) when $\ell=35\delta^2+48\delta+16$ has an integer solution $\delta$, (c) when $\ell=35\delta^2+8\delta$ has an integer solution $\delta$, and (d) when $\ell=2$ (Proposition \ref{prop:changemaker_application_2}).}
\label{fig:embedding_X(p_l,q_l)3}
\end{figure}
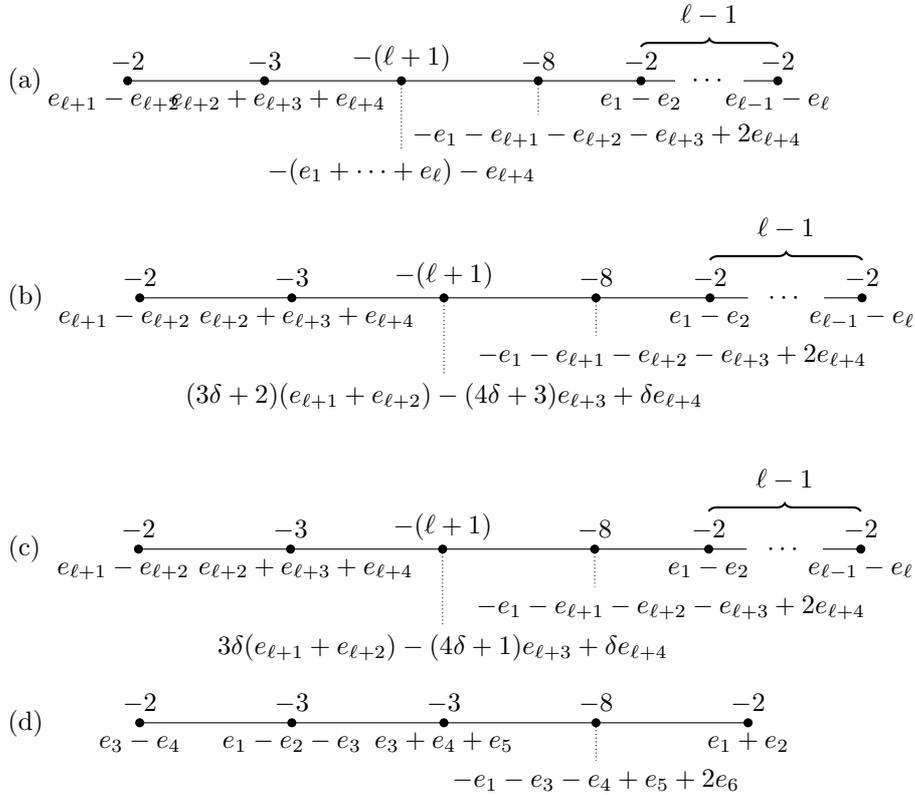

Note that the embedding given in Figure \ref{fig:embedding_X(p_l,q_l)3} (a) also works for $\ell \leq 9$. For the finitely many cases $\ell=1,\dots,9$, it can be directly verified that the embedding shown in Figure \ref{fig:embedding_X(p_l,q_l)3} (a) is the unique embedding into $-\mathbb{Z}^{\ell+4}$ up to an automorphism, except when $\ell=2$ or $\ell=3$. For $\ell=3$, the embedding from Figure \ref{fig:embedding_X(p_l,q_l)3} (b) with $\delta=-1$ provides another embedding into $-\mathbb{Z}^{\ell+4}=-\mathbb{Z}^7$. It can be checked that no additional embeddings exist. For $\ell=2$, Figure \ref{fig:embedding_X(p_l,q_l)3} (d) shows an embedding into $-\mathbb{Z}^6$ whose orthogonal complement is generated by $2(e_1-e_2)+4(e_3+e_4)-8e_5+9e_6$. It is straightforward to verify that no further embeddings exist. Thus, even for the case $\ell\leq 9$, we conclude that there is no embedding $Q_{X(p_\ell,q_\ell)}\hookrightarrow -\mathbb{Z}^{\ell+4}$ whose orthogonal complement is generated by a changemaker. 
\end{proof}

Finally, we show that most members of this family cannot be obtained from a $\mathbb{Q}$-homology $\mathbb{CP}^2$ construction by proving that no $\mathbb{Q}$-homology $\mathbb{CP}^2$ exists with a unique cyclic singularity of type $(p_\ell,p_\ell-q_\ell)$.

\begin{proposition}
    For each $\ell\geq 1$, let $p_\ell=35\ell^2+21\ell+3$ and $q_\ell=21\ell^2+14\ell+2$ be relatively prime integers. Then there is no $\mathbb{Q}$-homology $\mathbb{CP}^2$ with a unique cyclic singularity of type $(p_\ell,p_\ell-q_\ell)$ if $84\ell^2+84\ell+25$ is not a square number.
\end{proposition}

\begin{proof}
Note that \[
\frac{p_\ell}{p_\ell-q_\ell}=\frac{35\ell^2+21\ell+3}{14\ell^2+7\ell+1}=\begin{cases}
    [3,4,2,2,2,2,2,2], & \textrm{if $\ell=1$}, \\
    \left[ 3,3,\left[2\right]^{\ell-2},3,2,2,2,2,2,\ell+1\right], & \textrm{if $\ell\geq 2$}.
\end{cases}
\]
Thus, if $S$ is a $\mathbb{Q}$-homology $\mathbb{CP}^2$ with a unique cyclic singularity of type $(p_\ell,p_\ell-q_\ell)$, then by Proposition \ref{prop:D_square}, \[
D=84\ell^2+84\ell+25
\]
must be a square number.
\end{proof} 

\begin{remark}
Note that $84\ell^2+84\ell+25$ is not a square number for most positive integers $\ell$. Indeed, for $\ell \leq 10^8$, $84\ell^2+84\ell+25$ is a square number for only $8$ values of $\ell$.
\end{remark}
\smallskip

\subsection{Ballinger's Examples}\label{sec:Ballinger}
In \cite{Ballinger-2022}, Ballinger presents an infinite family of lens spaces that bound a simply-connected, smooth $4$-manifold with $b_2=b_2^+=1$ which cannot be constructed using a single $0$- and $2$-handle. Although this alone does not guarantee that these lens spaces are excluded from $\LK$, it is asserted in \cite[below Theorem 3.4]{Ballinger-2022} that the changemaker criterion (Theorem \ref{thm:changemaker}) can be applied to show their exclusion from $\LK$. In this section, we first show that such simply-connected $4$-manifolds with $b_2=b_2^+=1$ bounded by these lens spaces can also be constructed using a $\mathbb{Q}$-homology $\mathbb{CP}^2$ approach. We then confirm Ballinger's assertion that these lens spaces are indeed not contained in $\LK$ by examining the corresponding lattice embeddings.

For integers $n,k$ with $1<k<n-1$, define two relatively prime integers $p_{n,k}$ and $q_{n,k}$ by
\[p_{n,k}:=16n^2k-16nk^2-12n^2+4k^2+8n-2 \ \ \mathrm{and} \ \ q_{n,k}:=16nk-16k^2-12n+4k+5, \]
so that \[
\frac{p_{n,k}}{q_{n,k}}=\left[n,5,\left[2\right]^{n-k-2},6,\left[2\right]^{k-2} \right].
\]

\begin{theorem}[{\cite[Theorem 1.2, Theorem 3.4]{Ballinger-2022}}]\label{thm:Ballinger} If $(2k-1,2n-1)=1$, then the lens space $L(p_{n,k},q_{n,k})$ bounds a compact, oriented, simply-connected, smooth $4$-manifold $V_{n,k}$ with $b_2(V_{n,k})=b_2^+(V_{n,k})=1$ which cannot be built from a single $0$- and $2$-handle.     
\end{theorem}

The manifold $V_{n,k}$ is constructed as follows: Start with three generic projective lines in $\mathbb{CP}^2$. Then resolve their intersections repeatedly by taking connected sums with additional copies of $\mathbb{CP}^2$ (see \cite[p.45]{Gompf-Stipsicz-1999}). This process results in an embedding of a ring-shaped plumbing of $n+2$ spheres in the connected sum $\#n\mathbb{CP}^2$. Next remove one of the spheres from this configuration and smooth two of the remaining intersection points (see \cite[p.38]{Gompf-Stipsicz-1999}). This yields a linear plumbing of $n-1$ spheres embedded in $\#n\mathbb{CP}^2$, where the complement of a neighborhood of this new configuration has $b_2=1$ and a lens space boundary. Additionally, under certain conditions, the complement is simply-connected \cite[Proposition 3.1]{Ballinger-2022}. We notice that while this construction bears some resemblance to our $\mathbb{Q}$-homology $\mathbb{CP}^2$ construction, it cannot be achieved in the complex category.

We first show that simply-connected $4$-manifolds with $b_2=b_2^+=1$ bounded by the lens space $L(p_{n,k},q_{n,k})$ can also be obtained from a $\mathbb{Q}$-homology $\mathbb{CP}^2$ construction. 

\begin{proposition}\label{prop:Ballinger_construction} If $(2k-1,2n-1)=1$, then the lens space $L(p_{n,k},q_{n,k})$ bounds a compact, oriented, simply-connected, smooth $4$-manifold $W_{n,k}$ with $b_2(W_{n,k})=b_2^+(W_{n,k})=1$ which is the complement of a cone neighborhood of a singularity in a $\mathbb{Q}$-homology $\mathbb{CP}^2$.
\end{proposition}
\begin{proof} Note that \[
\frac{p_{n,k}}{p_{n,k}-q_{n,k}}=\left[ \left[2\right]^{n-2},3,2,2,n-k+1,2,2,2,k \right].
\]

Consider a configuration formed by the union of two sections and two fibers in the Hirzebruch surface $\Sigma_{n-k+1}$ of degree $n-k+1$, as illustrated in Figure \ref{fig:Ballinger} (a). By blowing up each of the two marked intersection points twice, we obtain a configuration of rational curves depicted in Figure \ref{fig:Ballinger} (b). And then, blow up the marked point $n-1$ times, followed by a single blow-up at each of the final two (-1)-curves that are disjoint from the sections, so that we get a configuration of rational curves shown in Figure \ref{fig:Ballinger} (c).

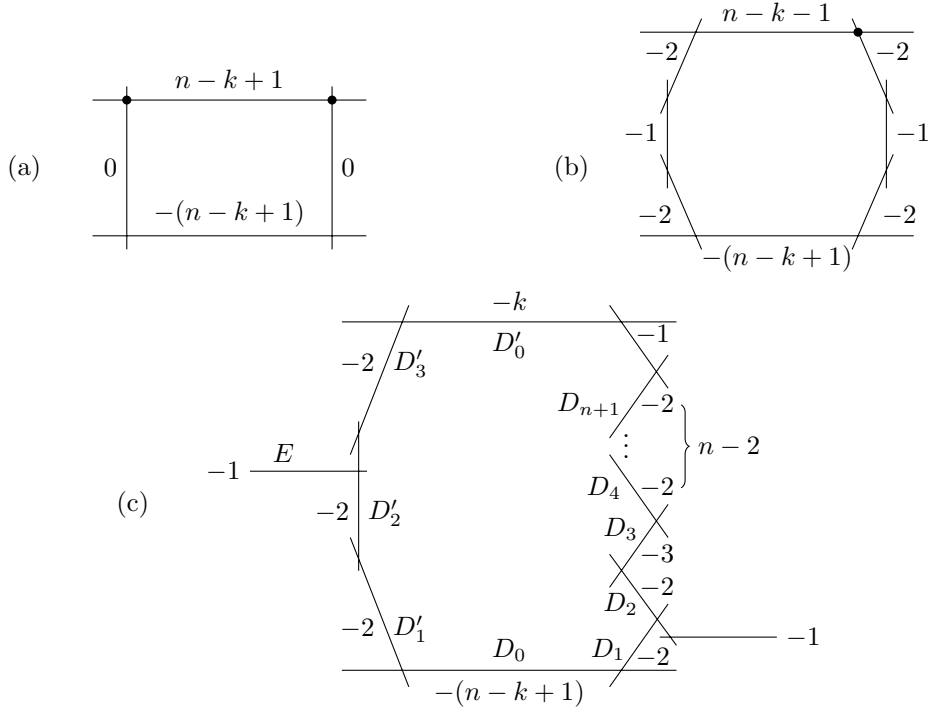
\begin{figure}[h]
\centering
\begin{tikzpicture}[scale=.9]
    \node at (-3,1) {(a)};
    \draw (-2,0)--(2,0) (1.5,-0.2)--(1.5,2.2)  (-1.5,-0.2)--(-1.5,2.2) (-2,2)--(2,2) ;
    \draw (0,0) node[above]{$-(n-k+1)$};
    \draw (0,2) node[above]{$n-k+1$};
    \draw (-1.5,1) node[left]{$0$};
    \draw (1.5,1) node[right]{$0$};
    \draw (1.5,2) node[circle, fill, inner sep=1.2pt, black]{};
    \draw (-1.5,2) node[circle, fill, inner sep=1.2pt, black]{};

    \begin{scope}[shift={(2,0)}]
    \node at (3,1) {(b)};
    \draw (4,0)--(8,0) (7.1,-0.2)--(7.7,1.2) (4.3,1.2)--(4.9,-0.2) (7.7,1.8)--(7.1,3.2) (4.3,1.8)--(4.9,3.2) (4.4,0.7)--(4.4,2.3) (7.6,0.7)--(7.6,2.3) (4,3)--(8,3);
    \draw (6,0) node[below]{$-(n-k+1)$};
    \draw (6,3) node[above]{$n-k-1$};
    \draw (4.6,0.3) node[left]{$-2$};
    \draw (7.4,0.3) node[right]{$-2$};
    \draw (4.4,1.5) node[left]{$-1$};
    \draw (7.6,1.5) node[right]{$-1$};
    \draw (4.7,2.7) node[left]{$-2$};
    \draw (7.3,2.7) node[right]{$-2$};
    \draw (7.185714,3) node[circle, fill, inner sep=1.2pt, black]{};
    \end{scope}
\end{tikzpicture}

\begin{tikzpicture}[scale=1.1]
    \node at (-5.5,2) {(c)};
        \draw (-3,0)--(1,0) (0.2,-0.2)--(0.9,0.8) (1,0.3)--(0.2,1.4) (0.2,1)--(0.9,2)        (-2.2,-0.2)--(-2.9,1.6) (-2.8,1)--(-2.8,3.2) (-2.9,2.6)--(-2.2,4.4)  (0.9,1.6)--(0.2,2.6) (0.2,2.8)--(0.9,3.8) (0.8,0.4)--(2.2,0.4) (-2.7,2.4)--(-4.1,2.4) (0.9,3.4)--(0.2,4.4) (-3,4.2)--(1,4.2);
\draw (0.4,2.8) node{$\vdots$};
\draw (-2.5,0.5) node[left]{$-2$};
\draw (-1,0) node[below]{$-(n-k+1)$};
\draw (-1,4.2) node[above]{$-k$};
\draw (2.2,0.4) node[right]{$-1$};
\draw (-4.1,2.4) node[left]{$-1$};
\draw (0.4,0.15) node[right]{$-2$};
\draw (0.45,1) node[right]{$-2$};
\draw (0.45,1.4) node[right]{$-3$};
\draw (-2.8,1.9) node[left]{$-2$};
\draw (-2.5,3.7) node[left]{$-2$};
\draw (0.45,3.2) node[right]{$-2$};
\draw (0.45,2.2) node[right]{$-2$};
\draw (0.4,4.03) node[right]{$-1$};

\draw (-1,4.2) node[below]{$D_0'$};
\draw (-1,0) node[above]{$D_0$};
\draw (-2.5,0.5) node[right]{$D_1'$};
\draw (-2.8,1.9) node[right]{$D_2'$};
\draw (-3.7,2.4) node[above]{$E$};
\draw (-2.5,3.7) node[right]{$D_3'$};
\draw (0.5,0.25) node[left]{$D_1$};
\draw (0.65,0.8) node[left]{$D_2$};
\draw (0.65,1.7) node[left]{$D_3$};
\draw (0.45,2.2) node[left]{$D_4$};
\draw (0.45,3.2) node[left]{$D_{n+1}$};
\draw [decorate,decoration={brace,amplitude=3pt},xshift=3pt]
	(0.95,3.2) -- (0.95,2.2) node [black,midway,xshift=18pt,yshift=0pt] 
	{$n-2$};
    \end{tikzpicture}
    \caption{Configurations over $n+5$ blow-ups from Hirzebruch surface $\Sigma_{n-k+1}$.}
    \label{fig:Ballinger}
\end{figure}

Let $\tilde{S}_{n,k}$ denote the resulting surface obtained by performing a total of $n+5$ blowing-ups on $\Sigma_{n-k+1}$. Label the rational curves as depicted in Figure \ref{fig:Ballinger} (c), and define \[
D:=D_0'+(D_1'+D_2'+D_3')+D_0+(D_1+\cdots+D_{n+1}).\] 
Let $\pi\colon\tilde{S}_{n,k}\to S_{n,k}$ denote the contraction of $D$. Then $S_{n,k}$ is a $\mathbb{Q}$-homology $\mathbb{CP}^2$ with a unique cyclic singularity of type $(p_{n,k},p_{n,k}-q_{n,k})$, and $\tilde{S}_{n,k}$ is its minimal resolution. Let $W_{n,k}$ denote the complement in $S_{n,k}$ of the cone neighborhood of the singularity. This is a smooth $4$-manifold with $b_2=b_2^+=1$ and $\partial W_{n,k}=L(p_{n,k},q_{n,k})$. 

We now show that $\pi_1(W_{n,k})=1$ using the argument provided in the proof of \cite[Theorem 3]{Lee-Park-2007} (see also \cite[Lemma 2.5]{Stipsicz-Szabo-2005}): Note that $\tilde{S}_{n,k}=W_{n,k}\cup_\partial X_{n,k}$, where $X_{n,k}:=X(p_{n,k},p_{n,k}-q_{n,k})$ is the plumbed $4$-manifold described in Section \ref{sec:preliminaries}. Since both $\tilde{S}_{n,k}$ and $X_{n,k}$ are simply-connected, Seifert-Van Kampen's theorem implies that $\pi_1(W_{n,k})/N =1$, where $N$ is the normal subgroup of $\pi_1(W_{n,k})$ generated by the image of the homomorphism $\pi_1(\partial W_{n,k})\to \pi_1(W_{n,k})$ induced by the natural inclusion. Observe that $\pi_1(\partial W_{n,k})=\pi_1(L(p_{n,k},p_{n,k}-q_{n,k}))=\mathbb{Z}_{p_{n,k}}$ is cyclic, and that the circle $C:=E\cap \partial W_{n,k}$, which is normal to the sphere $D_2'$, represents $2k-1$ times a generator of $\pi_1(\partial W_{n,k})$. On the other hand, we have \[
p_{n,k}=(2 k-1) (8n^2-8nk-4n+2k+1) -(2n-1)^2,
\]
and it follows that $C$ represents a generator of $\pi_1(\partial W_{n,k})$ assuming $(2k-1,2n-1)=1$. Thus, $N$ is equal to the normal subgroup of $\pi_1(W_{n,k})$ generated by the class of $C$. However, $C$ is homotopically trivial in $W_{n,k}$ since it can be contracted along the hemisphere $E\cap W_{n,k}$. Therefore, we conclude that $\pi_1(W_{n,k})=1$. \end{proof}

\begin{remark} As in the proof of Proposition \ref{prop:construction2}, one can show that \[
H_1(W_{n,k};\mathbb{Z})= \mathbb{Z}_{(2k-1,2n-1)}.
\]
Thus, the manifold $W_{n,k}$ described in Proposition \ref{prop:Ballinger_construction} is simply-connected if and only if $(2k-1,2n-1)=1$.
\end{remark}

Similar to Question \ref{question:example}, we may pose the following question. 
\begin{question} Are the $4$-manifolds $V_{n,k}$ in Theorem \ref{thm:Ballinger} and $W_{n,k}$ in Proposition \ref{prop:Ballinger_construction} diffeomorphic?     
\end{question}

Next, we show that the lens spaces of Theorem \ref{thm:Ballinger} are not contained in $\LK$ by applying the changemaker criterion (Theorem \ref{thm:changemaker}).

\begin{proposition}\label{prop:Ballinger_changemaker} If $(2k-1,2n-1)=1$, then the lens space $L(p_{n,k},q_{n,k})$ does not bound a compact, oriented, smooth $4$-manifold with $b_2=b_2^+=1$ built from a single $0$- and $2$-handle.     
\end{proposition}
\begin{proof} For convenience, let us denote $X(p_{n,k},q_{n,k})$ by $X_{n,k}$. We will present all lattice embeddings of $Q_{X_{n,k}}$ into $-\mathbb{Z}^{b_2(X_{n,k})+1}=-\mathbb{Z}^n$ (up to automorphism) along with the generators of the orthogonal complements. The arguments provided in Proposition \ref{prop:changemaker_application} and Proposition \ref{prop:changemaker_application_2} can be applied to show that these are indeed all embeddings $Q_{X_{n,k}}\hookrightarrow -\mathbb{Z}^n$.

\textbf{Case 1:} $k=2$. If $n=4$, then there are two embeddings as shown in Figure \ref{fig:embedding_Ballinger} (a). The corresponding orthogonal complements are generated by $e_1+3e_2+6e_3-8e_4$ and $2e_1-3e_2+4e_3-9e_4$, respectively. Neither of these vectors can be mapped to a changemaker under any automorphism of $-\mathbb{Z}^4$. For $n\geq 5$, Figure \ref{fig:embedding_Ballinger} (b) illustrates an embedding $Q_{X_{n,2}}\hookrightarrow -\mathbb{Z}^n$. Assuming $n\not\equiv 2 \mod 3$ (which follows from $(2k-1,2n-1)=1$), the orthogonal complement is generated by the vector \[
6(e_1+\cdots+e_{n-3})+3e_{n-2}+(2n-7)e_{n-1}-(4n-8)e_n,
\]
which cannot be mapped to a changemaker under any automorphism of $-\mathbb{Z}^n$. The embedding shown in Figure \ref{fig:embedding_Ballinger}(b) is the only embedding $Q_{X_{n,2}}\hookrightarrow -\mathbb{Z}^n$ unless $n = 7$. If $n=7$, then there is exactly one additional embedding as shown in Figure \ref{fig:embedding_Ballinger} (c), and the generator \[
e_1+e_2-e_3+e_4+7e_5-15e_6-18e_7
\]
of the orthogonal complement cannot be mapped to a changemaker under any automorphism of $-\mathbb{Z}^7$.

\textbf{Case 2:} $n=k+2\geq 5$. In this case, Figure \ref{fig:embedding_Ballinger} (d) illustrates an embedding $Q_{X_{n,k}}\hookrightarrow -\mathbb{Z}^n$. The orthogonal complement is generated by the vector \[
8(e_1+\cdots+e_{n-3})+(2n-5)e_{n-2}-(4n-10)e_{n-1}-(2n-9)e_n,
\]
which cannot be mapped to a changemaker under any automorphism of $-\mathbb{Z}^n$. This embedding is the only embedding $Q_{X_{n,k}}\hookrightarrow -\mathbb{Z}^n$ unless $n\neq 7$. If $n=7$, there is exactly one additional embedding as shown in Figure \ref{fig:embedding_Ballinger} (e), and the generator \[
e_5+2e_6-3e_7
\]
of the orthogonal complement cannot be mapped to a changemaker under any automorphism of $-\mathbb{Z}^7$.

\textbf{Case 3:} $k\geq 3$ and $n-k\geq 3$. In this case, Figure \ref{fig:embedding_Ballinger} (f) illustrates an embedding $Q_{X_{n,k}}\hookrightarrow -\mathbb{Z}^n$. Assuming $(2k-1,2n-1)=1$, the orthogonal complement is generated by the vector \[
(4k-2)(e_1+\cdots+e_{n-k-1})+(4n-4k)(e_{n-k}+\cdots+e_{n-2})+(2k-1)e_{n-1}+(2n-4k+1)e_n,
\]
which cannot be mapped to a changemaker under any automorphism of $-\mathbb{Z}^n$. This embedding is the only embedding $Q_{X_{n,k}}\hookrightarrow -\mathbb{Z}^n$ unless $(n,k)\in \{(7,3), (7,4), (8,3), (8,4)\}$. 

For $(n,k)=(7,3)$, there are exactly two additional embeddings as shown in Figure \ref{fig:embedding_Ballinger} (g), and the corresponding generators are \[
e_1-e_2+e_3+5e_4+11(e_5-e_6)-24e_7 \quad \textrm{and} \quad 3(e_1-e_2+e_3)-e_4+e_5-e_6-8e_7,
\]
respectively. The first vector cannot be mapped to a changemaker under any automorphism of $-\mathbb{Z}^7$. The second vector can actually be mapped to a changemaker, but it has square $-94\neq -846 = -p_{7,3}$.

For $(n,k)=(7,4)$, there is exactly one additional embedding as shown in Figure \ref{fig:embedding_Ballinger} (h), and the orthogonal complement is generated by \[
3(e_1+e_2-e_3)-5e_4-13(e_5-e_6)-22e_7,
\]
which cannot be mapped to a changemaker under any automorphism of $-\mathbb{Z}^7$.

For $(n,k)=(8,4)$, there is exactly one additional embedding as shown in Figure \ref{fig:embedding_Ballinger} (i), and the orthogonal complement is generated by \[
4(e_1+e_2-e_3)+5e_4-6(e_5-e_6+e_7)-35e_8,
\]
which cannot be mapped to a changemaker under any automorphism of $-\mathbb{Z}^8$.

Finally, for $(n,k)=(8,3)$, there are exactly two additional embeddings. However, we do not need to consider this case because $(2k-1,2n-1)=(5,15)=5\neq 1$.
\end{proof} 

\begin{figure}[t]
\flushleft
\begin{tikzpicture}[scale=1.1]
\node at (-3, 0) {(a)};
\draw (-1.5,0) node[circle, fill, inner sep=1.2pt, black]{};
\draw (0.5,0) node[circle, fill, inner sep=1.2pt, black]{};
\draw (2.5,0) node[circle, fill, inner sep=1.2pt, black]{};
\draw (5,0) node[circle, fill, inner sep=1.2pt, black]{};
\draw (7,0) node[circle, fill, inner sep=1.2pt, black]{};
\draw (9,0) node[circle, fill, inner sep=1.2pt, black]{};

\draw (-1.5,0) node[above]{$-4$};
\draw (0.5,0) node[above]{$-5$};
\draw (2.5,0) node[above]{$-6$};
\draw (5,0) node[above]{$-4$};
\draw (7,0) node[above]{$-5$};
\draw (9,0) node[above]{$-6$};

\draw (-1,-0.5) node[below]{$e_1-e_2-e_3-e_4$};
\draw (0.5,0) node[below]{$2e_2-e_3$};
\draw (2.7,0) node[below]{$2e_1+e_3+e_4$};

\draw (5.5,-0.5) node[below]{$e_1-e_2+e_3+e_4$};
\draw (7,0) node[below]{$-2e_1+e_3$};
\draw (9.2,0) node[below]{$e_1+2e_2+e_3$};

\draw[densely dotted] (-1.5,0)--(-1.5,-0.5) (5,0)--(5,-0.5);
\draw (-1.5,0)--(2.5,0) (5,0)--(9,0) ;
\end{tikzpicture}\vspace{5pt}

\begin{tikzpicture}[scale=1.2]
\node at (-3, 0) {(b)};
\draw (-1.5,0) node[circle, fill, inner sep=1.2pt, black]{};
\draw (0.5,0) node[circle, fill, inner sep=1.2pt, black]{};
\draw (2.5,0) node[circle, fill, inner sep=1.2pt, black]{};
\draw (4.5,0) node[circle, fill, inner sep=1.2pt, black]{};
\draw (6.5,0) node[circle, fill, inner sep=1.2pt, black]{};

\draw (-1.5,0) node[above]{$-n$};
\draw (0.5,0) node[above]{$-5$};
\draw (2.5,0) node[above]{$-2$};
\draw (4.5,0) node[above]{$-2$};
\draw (6.5,0) node[above]{$-6$};

\draw (-0.5,-0.5) node[below]{$-(e_1+\cdots+e_{n-2})+e_{n-1}-e_n$};
\draw (0.5,0) node[below]{$-e_1+2e_{n-2}$};
\draw (2.5,0) node[below]{$e_1-e_2$};
\draw (4.5,0) node[below]{$e_{n-4}-e_{n-3}$};
\draw (7,0) node[below]{$e_{n-3}+2e_{n-1}+e_n$};

\draw[densely dotted] (-1.5,0)--(-1.5,-0.5) ;
\draw (3.5,0) node{$\cdots$};
\draw [thick,decorate,decoration={brace,amplitude=3pt},xshift=0pt]
	(2.5,0.5) -- (4.5,0.5) node [black,midway,xshift=0pt,yshift=12pt] 
	{$n-4$};
\draw (-1.5,0)--(3,0) (4,0)--(6.5,0) ;
\end{tikzpicture}\vspace{10pt}

\begin{tikzpicture}[scale=1.1]
\node at (-3, 0) {(c)};
\draw (-1.5,0) node[circle, fill, inner sep=1.2pt, black]{};
\draw (0.5,0) node[circle, fill, inner sep=1.2pt, black]{};
\draw (2.5,0) node[circle, fill, inner sep=1.2pt, black]{};
\draw (4.5,0) node[circle, fill, inner sep=1.2pt, black]{};
\draw (6.5,0) node[circle, fill, inner sep=1.2pt, black]{};
\draw (8.5,0) node[circle, fill, inner sep=1.2pt, black]{};

\draw (-1.5,0) node[above]{$-7$};
\draw (0.5,0) node[above]{$-5$};
\draw (2.5,0) node[above]{$-2$};
\draw (4.5,0) node[above]{$-2$};
\draw (6.5,0) node[above]{$-2$};
\draw (8.5,0) node[above]{$-6$};

\draw (-0.5,-0.5) node[below]{$-e_1-e_2+e_3-e_4+e_5-e_6+e_7$};
\draw (0.3,0) node[below]{$e_1+e_2-e_3-e_6+e_7$};
\draw (2.8,0) node[below]{$-e_1+e_4$};
\draw (4.5,0) node[below]{$e_1-e_2$};
\draw (6.5,0) node[below]{$e_2+e_3$};
\draw (8.5,0) node[below]{$-e_3+2e_5+e_6$};

\draw[densely dotted] (-1.5,0)--(-1.5,-0.5) ;
\draw (-1.5,0)--(8.5,0) ;
\end{tikzpicture}

\begin{tikzpicture}[scale=1.2]
\node at (-3, 0) {(d)};
\draw (-1.5,0) node[circle, fill, inner sep=1.2pt, black]{};
\draw (0.5,0) node[circle, fill, inner sep=1.2pt, black]{};
\draw (2.5,0) node[circle, fill, inner sep=1.2pt, black]{};
\draw (4.5,0) node[circle, fill, inner sep=1.2pt, black]{};
\draw (6.5,0) node[circle, fill, inner sep=1.2pt, black]{};

\draw (-1.5,0) node[above]{$-n$};
\draw (0.5,0) node[above]{$-5$};
\draw (2.5,0) node[above]{$-6$};
\draw (4.5,0) node[above]{$-2$};
\draw (6.5,0) node[above]{$-2$};

\draw (0,-0.5) node[below]{$(e_1+\cdots+e_{n-3})-e_{n-2}+e_{n-1}+e_n$};
\draw (0.2,0) node[below]{$2e_{n-2}+e_{n-1}$};
\draw (2.6,0) node[below]{$-e_1-e_{n-1}+2e_n$};
\draw (4.5,0) node[below]{$e_1-e_2$};
\draw (6.5,0) node[below]{$e_{n-4}-e_{n-3}$};

\draw[densely dotted] (-1.5,0)--(-1.5,-0.5) ;
\draw (5.5,0) node{$\cdots$};
\draw [thick,decorate,decoration={brace,amplitude=3pt},xshift=0pt]
	(4.5,0.5) -- (6.5,0.5) node [black,midway,xshift=0pt,yshift=12pt] 
	{$n-4$};
\draw (-1.5,0)--(5,0) (6,0)--(6.5,0) ;
 \end{tikzpicture}\vspace{20pt}

\begin{tikzpicture}[scale=1.1]
\node at (-3, 0) {(e)};
\draw (-1.5,0) node[circle, fill, inner sep=1.2pt, black]{};
\draw (0.5,0) node[circle, fill, inner sep=1.2pt, black]{};
\draw (2.5,0) node[circle, fill, inner sep=1.2pt, black]{};
\draw (4.5,0) node[circle, fill, inner sep=1.2pt, black]{};
\draw (6.5,0) node[circle, fill, inner sep=1.2pt, black]{};
\draw (8.5,0) node[circle, fill, inner sep=1.2pt, black]{};

\draw (-1.5,0) node[above]{$-7$};
\draw (0.5,0) node[above]{$-5$};
\draw (2.5,0) node[above]{$-6$};
\draw (4.5,0) node[above]{$-2$};
\draw (6.5,0) node[above]{$-2$};
\draw (8.5,0) node[above]{$-2$};

\draw (-0.5,-0.5) node[below]{$-e_1-e_2+e_3-e_4+e_5+e_6+e_7$};
\draw (0.5,0) node[below]{$-2e_5+e_6$};
\draw (4.5,-0.5) node[below]{$e_1+e_2-e_3+e_5+e_6+e_7$};
\draw (4.5,0) node[below]{$-e_1+e_4$};
\draw (6.5,0) node[below]{$e_1-e_2$};
\draw (8.5,0) node[below]{$e_2+e_3$};

\draw[densely dotted] (-1.5,0)--(-1.5,-0.5) (2.5,0)--(3,-0.5) ;
\draw (-1.5,0)--(8.5,0) ;
\end{tikzpicture}\vspace{10pt}

\begin{tikzpicture}[scale=1]
\node at (-3,0) {(f)};
\draw (-1.5,0) node[circle, fill, inner sep=1.2pt, black]{};
\draw (0.5,0) node[circle, fill, inner sep=1.2pt, black]{};
\draw (2.5,0) node[circle, fill, inner sep=1.2pt, black]{};
\draw (4.5,0) node[circle, fill, inner sep=1.2pt, black]{};
\draw (6.5,0) node[circle, fill, inner sep=1.2pt, black]{};
\draw (8.5,0) node[circle, fill, inner sep=1.2pt, black]{};
\draw (10.5,0) node[circle, fill, inner sep=1.2pt, black]{};

\draw (-1.5,0) node[above]{$-n$};
\draw (0.5,0) node[above]{$-5$};
\draw (2.5,0) node[above]{$-2$};
\draw (4.5,0) node[above]{$-2$};
\draw (6.5,0) node[above]{$-6$};
\draw (8.5,0) node[above]{$-2$};
\draw (10.5,0) node[above]{$-2$};

\draw (1.5,-1) node[below]{$-(e_1+\cdots+e_{n-k-1})+(e_{n-k}+\cdots+e_{n-2})-e_{n-1}+e_n$};
\draw (0.5,0) node[below]{$-e_1+2e_{n-1}$};
\draw (2.5,0) node[below]{$e_1-e_2$};
\draw (4.5,-0.5) node[below]{$e_{n-k-2}-e_{n-k-1}$};
\draw (6.5,0) node[below]{$e_{n-k-1}-e_{n-k}+2e_n$};
\draw (8.5,-0.5) node[below]{$e_{n-k}-e_{n-k+1}$};
\draw (10.5,0) node[below]{$e_{n-3}-e_{n-2}$};

\draw[densely dotted] (-1.5,0)--(-1.5,-1) (4.5,0)--(4.5,-0.5) (8.5,0)--(8.5,-0.5);
\draw (3.5,0) node{$\cdots$};
\draw (9.5,0) node{$\cdots$};
\draw [thick,decorate,decoration={brace,amplitude=3pt},xshift=0pt]
	(2.5,0.5) -- (4.5,0.5) node [black,midway,xshift=0pt,yshift=12pt] 
	{$n-k-2$};
 \draw [thick,decorate,decoration={brace,amplitude=3pt},xshift=0pt]
	(8.5,0.5) -- (10.5,0.5) node [black,midway,xshift=0pt,yshift=12pt] 
	{$k-2$};
\draw (-1.5,0)--(3,0) (4,0)--(9,0) (10,0)--(10.5,0);
\end{tikzpicture}\vspace{15pt}

\begin{tikzpicture}[scale=1.1]
\node at (-3, 0) {(g)};
\draw (-1.5,0) node[circle, fill, inner sep=1.2pt, black]{};
\draw (0.5,0) node[circle, fill, inner sep=1.2pt, black]{};
\draw (2.5,0) node[circle, fill, inner sep=1.2pt, black]{};
\draw (4.5,0) node[circle, fill, inner sep=1.2pt, black]{};
\draw (6.5,0) node[circle, fill, inner sep=1.2pt, black]{};
\draw (8.5,0) node[circle, fill, inner sep=1.2pt, black]{};

\draw (-1.5,0) node[above]{$-7$};
\draw (0.5,0) node[above]{$-5$};
\draw (2.5,0) node[above]{$-2$};
\draw (4.5,0) node[above]{$-2$};
\draw (6.5,0) node[above]{$-6$};
\draw (8.5,0) node[above]{$-2$};

\draw (-0.5,-0.5) node[below]{$e_1-e_2+e_3-e_4-e_5+e_6-e_7$};
\draw (0.5,0) node[below]{$e_1-e_2+e_5-e_6+e_7$};
\draw (3,-0.5) node[below]{$-e_1+e_3$};
\draw (4.5,0) node[below]{$e_1+e_2$};
\draw (6.5,0) node[below]{$-e_2+2e_4-e_5$};
\draw (8.5,0) node[below]{$e_5+e_6$};

\draw[densely dotted] (-1.5,0)--(-1.5,-0.5) (2.5,0)--(3,-0.5) ;
\draw (-1.5,0)--(8.5,0) ;

\draw (-1.5,-1.5) node[circle, fill, inner sep=1.2pt, black]{};
\draw (0.5,-1.5) node[circle, fill, inner sep=1.2pt, black]{};
\draw (2.5,-1.5) node[circle, fill, inner sep=1.2pt, black]{};
\draw (4.5,-1.5) node[circle, fill, inner sep=1.2pt, black]{};
\draw (6.5,-1.5) node[circle, fill, inner sep=1.2pt, black]{};
\draw (8.5,-1.5) node[circle, fill, inner sep=1.2pt, black]{};

\draw (-1.5,-1.5) node[above]{$-7$};
\draw (0.5,-1.5) node[above]{$-5$};
\draw (2.5,-1.5) node[above]{$-2$};
\draw (4.5,-1.5) node[above]{$-2$};
\draw (6.5,-1.5) node[above]{$-6$};
\draw (8.5,-1.5) node[above]{$-2$};

\draw (-0.5,-2) node[below]{$-e_1+e_2-e_3+e_4+e_5-e_6-e_7$};
\draw (0.5,-1.5) node[below]{$e_1-e_2+e_5-e_6+e_7$};
\draw (3,-2) node[below]{$-e_1+e_3$};
\draw (4.5,-1.5) node[below]{$e_1+e_2$};
\draw (6.5,-1.5) node[below]{$-e_2+2e_4-e_5$};
\draw (8.5,-1.5) node[below]{$e_5+e_6$};

\draw[densely dotted] (-1.5,-1.5)--(-1.5,-2) (2.5,-1.5)--(3,-2) ;
\draw (-1.5,-1.5)--(8.5,-1.5) ;
\end{tikzpicture}\vspace{15pt}

\begin{tikzpicture}[scale=1.1]
\node at (-3, 0) {(h)};
\draw (-1.5,0) node[circle, fill, inner sep=1.2pt, black]{};
\draw (0.5,0) node[circle, fill, inner sep=1.2pt, black]{};
\draw (2.5,0) node[circle, fill, inner sep=1.2pt, black]{};
\draw (4.5,0) node[circle, fill, inner sep=1.2pt, black]{};
\draw (6.5,0) node[circle, fill, inner sep=1.2pt, black]{};
\draw (8.5,0) node[circle, fill, inner sep=1.2pt, black]{};

\draw (-1.5,0) node[above]{$-7$};
\draw (0.5,0) node[above]{$-5$};
\draw (2.5,0) node[above]{$-2$};
\draw (4.5,0) node[above]{$-6$};
\draw (6.5,0) node[above]{$-2$};
\draw (8.5,0) node[above]{$-2$};

\draw (-0.5,-0.5) node[below]{$-e_1-e_2+e_3-e_4-e_5+e_6+e_7$};
\draw (0.5,0) node[below]{$e_1+e_2-e_3-e_5+e_7$};
\draw (3,-0.5) node[below]{$e_5+e_6$};
\draw (4.5,0) node[below]{$-e_1+2e_4-e_5$};
\draw (6.5,0) node[below]{$e_1-e_2$};
\draw (8.5,0) node[below]{$e_2+e_3$};

\draw[densely dotted] (-1.5,0)--(-1.5,-0.5) (2.5,0)--(3,-0.5) ;
\draw (-1.5,0)--(8.5,0) ;
\end{tikzpicture}\vspace{18pt}

\begin{tikzpicture}[scale=1]
\node at (-3, 0) {(i)};
\draw (-1.5,0) node[circle, fill, inner sep=1.2pt, black]{};
\draw (0.5,0) node[circle, fill, inner sep=1.2pt, black]{};
\draw (2.5,0) node[circle, fill, inner sep=1.2pt, black]{};
\draw (4.5,0) node[circle, fill, inner sep=1.2pt, black]{};
\draw (6.5,0) node[circle, fill, inner sep=1.2pt, black]{};
\draw (8.5,0) node[circle, fill, inner sep=1.2pt, black]{};
\draw (10.5,0) node[circle, fill, inner sep=1.2pt, black]{};

\draw (-1.5,0) node[above]{$-8$};
\draw (0.5,0) node[above]{$-5$};
\draw (2.5,0) node[above]{$-2$};
\draw (4.5,0) node[above]{$-2$};
\draw (6.5,0) node[above]{$-6$};
\draw (8.5,0) node[above]{$-2$};
\draw (10.5,0) node[above]{$-2$};

\draw (0.5,-0.5) node[below]{$-e_1-e_2+e_3-e_4+e_5-e_6+e_7+e_8$};
\draw (0.3,0) node[below]{$e_1+e_2-e_3+e_5-e_6$};
\draw (4.5,-0.5) node[below]{$-e_5+e_7$};
\draw (4.5,0) node[below]{$e_5+e_6$};
\draw (6.5,0) node[below]{$-e_1+2e_4-e_6$};
\draw (8.5,0) node[below]{$e_1-e_2$};
\draw (10.5,0) node[below]{$e_2+e_3$};

\draw[densely dotted] (-1.5,0)--(-1.5,-0.5) (2.5,0)--(3.7,-0.5) ;
\draw (-1.5,0)--(10.5,0) ;
\end{tikzpicture}
\caption{Embeddings of $Q_{X_{n,k}}$ into $-\mathbb{Z}^n$ in Proposition \ref{prop:Ballinger_changemaker}: (a) for $(n,k)=(4,2)$, (b) for $k=2$ and $n\geq 5$, (c) for $(n,k)=(7,2)$, (d) for $n=k+2\geq 5$, (e) for $(n,k)=(7,5)$, (f) for $k\geq 3$ and $n-k\geq 3$, (g) for $(n,k)=(7,3)$, (h) for $(n,k)=(7,4)$, and (i) for $(n,k)=(8,4)$.}
\label{fig:embedding_Ballinger}
\end{figure}
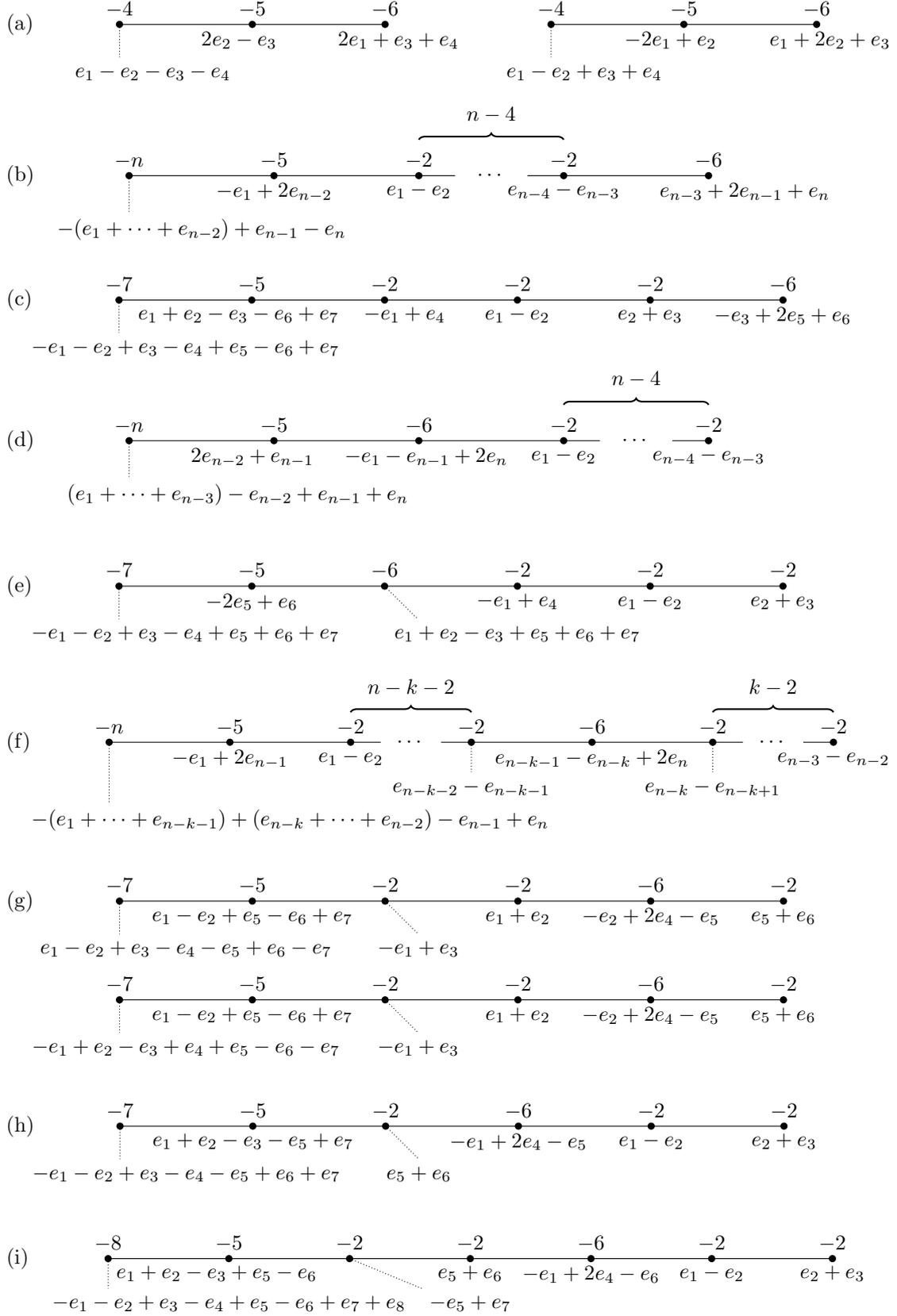

\clearpage
\bibliography{references}{}
\bibliographystyle{alpha}
\end{document}